\documentclass[12pt]{amsart}
\usepackage[utf8]{inputenc}

\usepackage[style=alphabetic,url=false]{biblatex}
\bibliography{refs} 

\usepackage{amsmath}
\usepackage{amstext}
\usepackage{amsfonts}
\usepackage{amssymb, dsfont}
\usepackage{amsthm}
\usepackage{bbm}
\usepackage{bm}
\usepackage{mathtools}
\usepackage{csquotes}
\usepackage{color}
\usepackage{comment}
\usepackage{graphicx}
\usepackage[colorlinks=true, allcolors=blue]{hyperref}
\usepackage{float}
\allowdisplaybreaks
\usepackage{tikz-cd}
\usepackage{amsbsy}
\usepackage{amscd}
\usetikzlibrary{matrix,arrows,decorations.pathmorphing}
\usepackage{enumitem}
\usepackage{setspace}
\usepackage{strdiag}

\usepackage[english]{babel}

\usepackage{a4wide}

\newcommand{\op}[1]{\operatorname{#1}}

\newtheorem{thm}{Theorem}[section]
\newtheorem{lem}[thm]{Lemma}

\newtheorem{prop}[thm]{Proposition}

\newtheorem*{nonothm}{Theorem}

\newtheorem{thmx}{Theorem}

\theoremstyle{definition}
\newtheorem{example}[thm]{Example}
\theoremstyle{definition}
\newtheorem{question}[thm]{Question}
\theoremstyle{definition}
\newtheorem{defn}[thm]{Definition}
\theoremstyle{definition}
\newtheorem{remark}[thm]{Remark}

\DeclareMathOperator{\Rep}{Rep}
\DeclareMathOperator{\Irr}{Irr}
\DeclareMathOperator{\Aut}{Aut}
\DeclareMathOperator{\Qut}{\mathbb{Q}ut}
\DeclareMathOperator{\Hom}{Hom}
\DeclareMathOperator{\Cay}{Cay}
\DeclareMathOperator{\Tr}{Tr}
\newcommand{\id}{\mathrm{id}}
\newcommand{\I}{{I}}

\newcommand{\G}{\mathbb{G}}
\newcommand{\Hb}{\mathbb{H}}
\newcommand{\X}{\mathbb{X}}
\newcommand{\C}{\mathbb{C}}
\newcommand{\N}{\mathbb{N}}

\newcommand{\HH}{\mathsf{H}}
\newcommand{\K}{\mathsf{K}}
\newcommand{\Gr}{\mathcal{G}}
\newcommand{\Hr}{\mathcal{H}}
\newcommand{\cO}{\mathcal{O}}

\newcommand{\Cst}{\textrm{C}^\ast}
\newcommand{\ol}[1]{\overline{#1}}

\newcommand{\mc}[1]{\mathcal{#1}}

\numberwithin{equation}{section}
\begin{document}

\title[Quantum Frucht's Theorem]{A quantum Frucht's theorem and quantum automorphisms of quantum Cayley graphs}
\author[Brannan]{Michael Brannan}
\address{Department of Pure Mathematics, Faculty of Mathematics, University of Waterloo. 200 University Ave W, Waterloo, ON N2L 3G1, Canada}
\email{michael.brannan@uwaterloo.ca}

\author[Gromada]{Daniel Gromada}
\address{Czech Technical University in Prague, Faculty of Electrical Engineering, Department of Mathematics, Technická 2, 166 27 Praha 6, Czechia}
\email{gromadan@fel.cvut.cz}

\author[Matsuda]{Junichiro Matsuda}
\address{Department of Pure Mathematics, Faculty of Mathematics, University of Waterloo. 200 University Ave W, Waterloo, ON N2L 3G1, Canada}
\email{junichiro.matsuda@uwaterloo.ca}

\author[Skalski]{Adam Skalski}
\address{Institute of Mathematics of the Polish Academy of Sciences, ul. \'Sniadeckich 8, 00-656 Warszawa, Poland}
\email{a.skalski@impan.pl}

\author[Wasilewski]{Mateusz Wasilewski}
\address{Institute of Mathematics of the Polish Academy of Sciences, ul. \'Sniadeckich 8, 00-656 Warszawa, Poland}
\email{mwasilewski@impan.pl}
\date{December 2024}

\begin{abstract}
We establish a quantum version of Frucht's Theorem, proving that every finite quantum group is the quantum automorphism group of an undirected finite quantum graph. The construction is based on first considering several quantum Cayley graphs of the quantum group in question, and then providing a method to systematically combine them into a single quantum graph with the right symmetry properties. We also show that the dual $\widehat \Gamma$ of any non-abelian finite group $\Gamma$ is ``quantum rigid''. 
 That is,  $\widehat \Gamma$ always admits a quantum Cayley graph whose quantum automorphism group is exactly $\widehat \Gamma$. 
\end{abstract}

\subjclass[2020]{Primary 46L89; Secondary 05C25, 20G42}
\keywords{Quantum graph; quantum automorphism groups; Frucht's theorem}

\maketitle

\section{Introduction and motivation}

The study of symmetries of graphs is a very natural and rich area of research, dating back to the origins of the graph theory, and obviously connected to the theory of permutation groups. The classical Cayley theorem says that every finite group is a subgroup of a finite permutation group (as it acts faithfully on itself), and the added combinatorial structure of graphs and their symmetries makes it possible to strengthen this result in a very satisfactory way. Indeed, let us recall the classical {\it Frucht's theorem} from \cite{Fru39}, which shows that the automorphism groups of graphs are universal among finite groups.

\begin{nonothm} [Frucht]
For every finite group $G$ there exists a finite simple graph $\Gr$ such that $G \cong \Aut(\Gr)$.    
\end{nonothm}

The main idea of the proof of Frucht's theorem is fairly straightforward. One first lists elements of our group, writing  $G = \{g_1, g_2, \ldots, g_n\}$, and defines a {\it directed}, {\it edge-coloured} graph $\Gr_0$ whose vertex set coincides with the set $G$ and the edge colours are given by $[n] = \{1, \ldots, n\}$. 
 An ordered pair $(g_i,g_j)$ is then an edge of $\Gr_0$ with colour $k \in [n]$ if $g_ig_j^{-1} = g_k$.  Note that $\Gr_0$ is simply a coloured version of the Cayley graph of $G$, with respect to the full generating set $S = G$. Alternatively, we can view each colour $k$ as representing a separate Cayley graph corresponding to the (not necessarily generating) set $S=\{g_k\}$. Next, one can go on to show that $G = \Aut\Gr_0$. We will recall the statement and its proof below as Proposition~\ref{P.classical1}.  In the second step one applies a procedure of creating a new simple (non-coloured) graph $\Gr$, whose automorphism group is the same as in case of $\Gr_0$. We recall Frucht's construction in Section~\ref{sec.step2} and also formulate a new,  alternative approach to proving this result in Propositions~\ref{P.classical21}, \ref{P.classical22}. 
 
 Our new approach to  Frucht's theorem comes with two distinct advantages.  Firstly, our construction is more efficient in the sense that it yields graphs $\Gr$ with much fewer vertices than those provided by Frucht's method.  Secondly, our construction turns out to be more amenable to {\it quantisation of Frucht's theorem}, which is the main goal of this paper.  More precisely, we wish to understand to what extent Frucht's theorem holds in the realm of {\it quantum symmetry groups} (see for example \cite{ASS17}).  
This problem was already addressed before, with certain non-universality results obtained. First, we may ask whether for every finite quantum group $\G$ (viewed as an object described by a finite-dimensional CQG-algebra, see \cite{DiK94}, fitting also in the theory of Woronowicz \cite{NT13}), there is a graph $\Gr$ such that $\G$ is its quantum automorphism group. This turns out to be false, as there are actually finite quantum groups not faithfully acting on any finite classical spaces -- in other words, that do not appear as (quantum) subgroups of \emph{quantum permutation groups} $S_N^+$ \cite{BBN12}. This can be reformulated by saying that a naive quantum version of the Cayley theorem fails. So, we might relax our expectation and ask, for any (finite) quantum subgroup of $S_N^+$, whether there is a graph that has the prescribed quantum automorphism group. Nevertheless, Banica and McCarthy showed in \cite{BM22} that this is false as well. For instance, the Kac--Paljutkin quantum group introduced in \cite{KP66} is finite, it acts faithfully on four points, but it is not an automorphism group of any finite graph.

From a certain viewpoint, the above no-go results are quite natural. Every finite group can be realized as a subgroup of a permutation group since it acts at least on itself. 
But for quantum groups, a problem appears already in this first step: every finite quantum group does act on itself, but it is itself not a classical set. Hence, in order to have a perfectly analogous situation for finite quantum groups, we must not only consider actions on sets, but also on {\it quantum sets}: finite-dimensional C*-algebras equipped with distinguished (tracial) states.  Within the framework of quantum sets, there is a well-understood notion of a {\it quantum graph} $\Gr$ over a quantum set -- see e.g.  \cite{MRV18, MRV19, Daw24, Was24}.  Moreover, one can construct the {\it quantum automorphism group of $\Gr$}, $\Qut(\Gr)$, which is a compact matrix quantum group which generalises the usual quantum automorphism group of a graph \cite{Daw24,MRV19, BCEHPSW20}.  In general $\Qut(\Gr)$ can be an infinite compact quantum group, even when $\Gr$ is classical.  This can easily be seen by taking $\Gr$ to be a complete (quantum) graph.  For more exotic examples, one can consider Hadamard graphs and related examples \cite{Gr24, Gr23}, the many examples considered by Schmidt in \cite{Sch20a, Sch20b}, or the Higman-Sims graph - the latter yielding a quantum automorphism group whose dual is a discrete quantum group with property (T) (see \cite{RV24}, where the authors point out that it can be deduced from a combination of results from \cite{Kuperberg97}, \cite{Arano18}, \cite{Edge19}).  
At this time, we do not know if every compact matrix quantum group arises as the quantum automorphism group of a quantum graph quantum graphs. 
Nevertheless, one can at least ask: Can every {\it finite} quantum group be realised as the quantum automorphism group of a quantum graph? The main result of our paper is that this is indeed the case, establishing a satisfactory generalisation of Frucht's theorem.

\begin{thmx}[Quantum Frucht's theorem, Theorem \ref{thm:qF}]
\label{T.qFrucht}
Let $\G$ be a finite quantum group.  Then there is an undirected finite quantum graph $\Gr$ such that $\G = \Qut(\Gr)$.  That is, $\G$ acts faithfully on $\Gr$, and the canonical quotient map $C(\Qut(\Gr)) \to C(\G)$ is an isomorphism of Hopf *-algebras.
\end{thmx}

For the proof of the quantum Frucht's theorem, it is natural to try and find a way to quantise Frucht's original arguments.  However, two issues arise. Firstly, given a finite quantum group $\G$, there seems to be no straightforward way to ensure the existence of a single coloured quantum graph $\Gr_0$ whose quantum automorphism group is exactly $\G$.  Secondly, even if such a $\Gr_0$ could be found, Frucht's decolouring procedure involves the addition of edges to $\Gr_0$ -- a process which seems quite hard to effectively quantise. Instead, we develop a new approach to Frucht's theorem which is much more amenable to quantisation of both groups and graphs.  We present the main proof of the quantum Frucht's theorem in Sections~\ref{sec.step1}--\ref{sec.step2}. As a first step, given any finite quantum group $\G$, we construct in Proposition \ref{P.quantum1} a finite set of quantum graphs $\{\Gr_i\}$, for each of which the quantum vertex set is $\G$ and such that $\G$ is the intersection of their quantum automorphism groups. As a second step, we establish Theorem~\ref{T.q22}, which provides a general procedure that allows one  to combine the family $\{\Gr_i\}$ into a single quantum graph $\Gr$ over a 
larger quantum vertex set which still has $\G$ as its quantum automorphism group. 

We remark here that there has been some recent partial progress towards proving a quantum Frucht's Theorem in other works. 
 For example, in \cite{DBRS24} it was shown that every classical finite group can be realized as the quantum automorphism group of a classical graph (the authors call this a ``weak quantum Frucht's theorem''). 
 The key idea in their approach was to check that each step of Frucht's original decolouration procedure does not introduce any new {\it quantum} symmetries.  In another direction, the work \cite{BH24} shows that certain finite quantum groupoids can be realized as quantum automorphism groupoids of ``categorified graphs''. The latter are certain variations of quantum graphs modelled over finite index subfactors.

In both the classical and quantum Frucht's theorems, a key role is played by (quantum) Cayley graphs and their symmetries.  Indeed, these are the natural objects which support quantum group symmetry.  An old problem in group theory, dating back to work of Watkins \cite{Wat76} (see also \cite{LdlS}), is to classify the finite (or infinite) groups $G$ which admit a Cayley graph $\Cay(G,S)$ with few automorphisms, meaning that $\Aut(\Cay(G,S)) = G$. Such Cayley graphs are also called {\it rigid}.  In Section \ref{sec:rigid-calyey}, we initiate the study of this problem for finite quantum groups, and establish some partial results there.  Most notably, we settle the problem completely for all duals of finite groups. 

 \begin{thmx}[Theorem \ref{Thm:cayleyaut}]
Let $\Gamma$ be a non-abelian finite group.  Then the Pontryagin dual quantum group $\widehat \Gamma$ admits a rigid quantum Cayley graph $\Gr$. 
In other words, there exists a (directed) quantum Cayley graph structure $\Gr$ on $\widehat \Gamma$ such that $\widehat \Gamma = \Qut(\Gr)$. 
 \end{thmx}

Applying the above result to finite perfect groups (i.e.\ those with trivial abelian quotients), we are able to easily generate many examples of quantum graphs with small classical automorphism group, and large quantum automorphism group.

 \begin{thmx}[Theorem \ref{thm:trivquat}]
 Let $\G= \widehat{\Gamma}$, where $\Gamma$ is a finite perfect group. Then there exists a quantum (Cayley) graph structure on $\G$ such that the classical automorphism group is trivial, but the quantum automorphism group acts transitively.
 \end{thmx}

The detailed plan of the paper is as follows: in Section \ref{sec:prelims}, we provide a brief introduction to quantum sets, graphs, and their quantum automorphism groups. Section \ref{sec.step1} is devoted to a brief introduction to (finite) quantum Cayley graphs and makes connections to earlier work on this topic due to Vergnioux \cite{Ver05}.  In Section \ref{sec:coloured-Frucht} we prove our coloured version of the quantum Frucht's theorem, and the decolouring and combining procedure is presented in Section \ref{sec.step2}.  The final Section \ref{sec:rigid-calyey} is devoted to the study of rigid quantum Cayley graphs, and also studies several concrete examples related to duals of symmetric groups.   
    
%

\section{Preliminaries} \label{sec:prelims}

In this section, we discuss preliminary facts regarding quantum sets, quantum graphs, and their quantum automorphism groups.

\subsection{Quantum sets and quantum graphs}



Classical and quantum finite graphs play a key role in our work. It will be convenient to present a uniform picture of both, so we begin with a rather non-standard approach to standard graphs.

Let $X$ be a finite set. When we equip it with the counting measure,  $\ell^2(X)$ is the Hilbert space with orthonormal basis $(e_x)_{x\in X}$ indexed by elements of $X$.
A \emph{(directed) graph} is then given by a pair $\Gr=(X,A)$ of a finite set $X$ of vertices and an adjacency matrix $A\colon \ell^2(X)\to \ell^2(X)$  whose entries are $0$ or $1$ (in the canonical basis). This condition is equivalent to $A\bullet A=A$, where $\bullet$ denotes the entrywise product of matrices, usually called the \emph{Schur product} or the \emph{Hadamard product}. The condition on the entries can sometimes be weakened, when weighted or coloured graphs are considered. Given $x,y \in X$ if $A_{yx}=1$, we say that there is an oriented edge $x\to y$ (note that most of the graph-theoretical literature uses a transposed convention). The graph is undirected if $A=A^*$, where ${}^*$ denotes the usual adjoint (the conjugate transpose). The graph admits no loops if the diagonal of $A$ is zero. If both these conditions are satisfied, the graph is called \emph{simple}.

If $G$ is a finite group, we denote by $\star$ the group operation extended to $\ell^2(G)$ linearly and call it the \emph{convolution}. 
 A Cayley graph of $G$ with respect to a set  $S\subseteq G$ is the graph $\Cay(G,S)=(G,A)$, where $A_{gh}=1$ if $g=kh$ for some $k\in S$ and zero otherwise. Equivalently, we can say that $Af=p \star f$ for any $f\in \ell^2(G)$, where $p=\sum_{k\in S} e_k$. Note that one usually assumes that $S$ is a generating subset of $G$.

The notion of quantum sets, sometimes called measured quantum sets, appeared for example in \cite{MRV18}. A \emph{finite quantum set} is a virtual object $\X$ associated to a finite-dimensional (possibly non-commutative) C*-algebra $C(\X)$, which is interpreted as the algebra of functions over $\X$. Given a tracial faithful positive functional $\psi$ on $C(\X)$, we may define an inner product $\langle x,y\rangle=\psi(x^*y)$, $x,y \in C(\X)$. Denote by  $m\colon C(\X)\otimes C(\X)\to C(\X)$  the multiplication map. It is well-known that there is a unique tracial positive functional $\psi:C(\X) \to \C$ such that considering the associated Hilbert space structure, we have $mm^*=\id$ (in other words, $C(\X)$ admits a unique \emph{tracial $1$-form}). We will always consider $\X$ with the tracial $1$-form $\psi$ and the associated inner product, and denote the particular Hilbert space by $\ell^2(\X)$. The multiplication in $C(\X)$ will be denoted $\cdot$, the involution $^*$ and the unit $1_{C(\X)}$ will be written as $\eta\in \ell^2(\X)$ when we want to interpret it as a Hilbert space vector. For any operators $A,B\colon \ell^2(\X)\to \ell^2(\X)$, we will define their \emph{Schur product} by the formula  $A\bullet B=m(A\otimes B)m^*$   and denote by $\bar A$ the operator acting by $\bar Af=(A( f^*))^*$ for any $f\in \ell^2(\X)$, calling it simply the \emph{complex conjugate} of $A$. We will save the $*$-symbol for denoting the adjoint operation in a C$^\ast$-algebra.  (In particular the $\ast$-operation on $\ell^2(\X)$ is simply the adjoint operation inherited from $C(\X) \cong \ell^2(\X)$). We will occasionally use diagrams to illustrate some complicated equations. For that purpose, we will use the following notation for the multiplication and unit: $m=\spider{2/1}$, $m^*=\spider{1/2}$, $\eta=\spider{0/1}$, $\spider{2/0}=\Diagram{\Dmor{bcirc}2/1 (1,0.5) \Dmor{bcirc}0/0 (1,1)}=\psi\circ m$. We can use this notation to express the Schur product as follows:
$$A\bullet B=\Diagram{\Dmor{bcirc}[-/-0-] (1,-1)\Dmor{bcirc}[-0-/-] (1,2) \DMor{square}1/1 (0,0.5) {$A$} \DMor {square}1/1 (2,0.5) {$B$}}.$$

A key example of a finite quantum set of importance for our paper is given by a \emph{finite quantum group}, i.e.\ a virtual object $\G$ associated to a finite-dimensional CQG Hopf *-algebra (so in particular a $\Cst$-algebra) denoted $C(\G)$ or $\cO(\G)$. Finite quantum groups are a subclass of compact quantum groups, which will be discussed below. The bi-invariant, with respect to the comultiplication, positive-functional of $\G$, the so-called \emph{Haar functional} $h$, is, when properly normalized, the tracial $1$-form on $C(\G)$, making $(C(\G), h)$ a (measured) finite quantum set. Indeed, by \cite[Proposition A 2.2]{Wor87} the Haar functional on $C(\G)$ is given by $\bigoplus_{\alpha} n_{\alpha} \op{Tr}_{n_{\alpha}}$, where $C(\G)\simeq \bigoplus_{\alpha} M_{n_{\alpha}}$.


For the history and motivation behind the notion of a quantum graph we refer to the articles \cite{Was24} and \cite{Daw24}, where also alternative equivalent approaches are described. We will mostly focus on the setup described via quantum adjacency matrices. Thus a \emph{(directed) quantum graph} is given by a pair $\Gr=(\X,A)$, where $\X$ is a finite quantum set and $A\colon \ell^2(\X)\to \ell^2(\X)$, the \emph{quantum adjacency matrix}, is a real Schur idempotent, i.e.\ it satisfies the equalities $A\bullet A=A=\bar A$. The quantum graph is called \emph{undirected} if $A=A^*$,  and it is said to have \emph{no loops} if $A\bullet \I=0$, where $\I$ stands for the identity. Note that if $A=\bar{A}$ then (as we work with a tracial reference functional) we always have $A^* = \bar{A^*}$.
Several examples will appear below; for now just note that the edgeless quantum graph on $\X$ is given by the map $A = 0$, and the complete quantum graph (without loops) by the map $Af = \psi(f) \eta-f,$ $f \in C(\X)$.

Given a directed quantum graph as above we define its \emph{(in)degree spectrum} to be the spectrum of the operator $D_A:=m(A\eta\otimes\I_{\ell^2(\X)}):\ell^2(\X) \to \ell^2(\X)$.  Likewise, we define the \emph{outdegree spectrum} via the spectrum of the operator $D^A=m(A^*\eta\otimes\I_{\ell^2(\X)})$. The eigenspace of $D_A$ corresponding to an eigenvalue $d\in \C$ is interpreted as functions supported on vertices of (in)degree $d$. The operators $D_A$ and $D^A$ clearly coincide if $\Gr$ is undirected. Following \cite[Definition 2.24]{Mat22} we call a quantum graph $(\X,A)$ \emph{$d$-regular} if $A\eta=A^*\eta=d\eta$.

\begin{prop}
  Let $d \in \C$.  A quantum graph $\Gr=(\X,A)$ is $d$-regular if and only if $D_A=D^A=d\,\I_{\ell^2(\X)}$.
\end{prop}
\begin{proof}If $(\X,A)$ is $d$-regular, then by definition $A\eta=d\eta$ and hence $D_A=m(A\eta\otimes\I_{\ell^2(\X)})=d\,m(\eta\otimes\I_{\ell^2(\X)})=d\,\I_{\ell^2(\X)}$. Similarly for $D^A$. On the other hand, if $D_A=d\,\I_{\ell^2(\X)}$, then $d\eta=D_A\eta=m(A\eta\otimes\eta)=A\eta$.
\end{proof}

\subsection{Quantum symmetry/automorphism groups}

In order to speak about quantum symmetries of (quantum) graphs, we actually need to introduce compact quantum groups because, in contrast with the classical case, even finite (quantum) graphs can admit infinitely many quantum symmetries. We will mostly consider an algebraic approach. A \emph{compact quantum group} in the sense of Woronowicz is a virtual object $\G$ associated to a \emph{CQG-Hopf} *-algebra $\cO(\G)$ (for a definition of a CQG-Hopf *-algebra see \cite{DiK94}). 
Practically all compact quantum groups studied in this paper will be \emph{compact matrix quantum groups}, i.e.\ will admit a unitary matrix $U=(u_{ij})_{i,j=1}^n \in M_n(\cO(\G))$ (for some $n \in \N)$ such that $U$ is a representation of $\G$, that is 
\[ \Delta(u_{ij}) = \sum_{k=1}^n u_{ik} \otimes u_{kj}, \;\;\; i, j=1, \ldots, n,\]
and the set of entries of $U$ generates $\cO(\G)$ as a $^*$-algebra. We will then call $U$ a \emph{fundamental} representation of $\G$.

Observe that any finite quantum group $\G$ as introduced in the previous subsection can be considered as a compact matrix quantum group. Indeed, the comultiplication can be interpreted as a map $\ell^2(\G)\to C(\G)\otimes \ell^2(\G)$, so as a matrix acting on $\ell^2(\G)$ with entries in $C(\G)$, i.e.\ as a fundamental representation of $\G$.

Given two compact quantum groups $\G$ and $\Hb$ we say that $\Hb$ is a \emph{(closed) quantum subgroup} of $\G$ if there exists a surjective Hopf $^*$-algebra morphism $\pi: \cO(\G)\to \cO(\Hb)$.
Similarly given matrix compact quantum groups $\G$ and $\Hb$ with fundamental representations of the same size, we define their \emph{intersection} $\G\cap\Hb$ to be given by the largest quotient $\cO(\G\cap\Hb)$ of both $\cO(\G)$ and $\cO(\Hb)$ mapping generators to generators. That is, $\cO(\G\cap\Hb)$ is the universal *-algebra generated by $u_{ij}$ subject to all the relations from both $\G$ and $\Hb$. One can check that this notion is compatible with the standard one, discussed for example
in \cite{CHK17}.

For a detailed description of actions of compact quantum groups we refer to \cite{DC17}. We will be solely concerned with actions on finite quantum spaces, where certain subtleties disappear and one can view the theory purely algebraically. So an \emph{action of a  compact quantum group $\Hb$ on a finite-dimensional $\Cst$-algebra} $C(\X)$ is an injective unital $^*$-homomorphism $\alpha: C(\X) \to C(\X) \otimes \cO(\Hb)$ satisfying the \emph{action equation}:
\[ (\alpha \circ \id_{\cO(\Hb)}) \alpha = (\id_{C(\X)} \otimes \Delta_{\Hb}) \circ \alpha.\]
The action $\alpha$ as above is said to be \emph{faithful} if $\cO(\Hb)$ is the Hopf $^*$-algebra generated by the set of slices $\{(\omega \otimes \id_{\cO(\Hb)})\circ \alpha(a): a \in C(\X), \omega \in C(\X)^*\}$. It is said to be \emph{ergodic} or \emph{transitive} if
the fixed point subalgebra is trivial: $\{x \in C(\X): \alpha(x) = x \otimes 1_{\cO(\Hb)}\} = \C 1_{C(\X)}$.

If $\X$ is a finite quantum set, so we have also a fixed positive faithful functional $\psi: C(\X) \to \C$, we say that the action as above \emph{preserves} $\psi$ if 
\[ (\psi \otimes \id_{\cO(\Hb)})\circ \alpha = \psi(\cdot) 1_{\cO(\Hb)}.\]
In \cite{Wan98} Wang showed that any finite quantum set $\X$ admits a \emph{quantum symmetry group} i.e.\ there exists a \emph{universal} compact quantum group $\Qut(\X, \psi)$ acting on $\X$ with the action preserving $\psi$; with the action automatically faithful. More information, in particular the meaning of universality can be found for example in \cite{ASS17}, but see also Proposition \ref{prop:Matt} below.

We will mention one specific example, which is relevant for what follows: Let $\G$ be a finite quantum group, and let $\psi$ be the suitable scalar multiple of the Haar state of $\G$ (so that $(C(\G), \psi)$ is a finite quantum set, as explained in the previous subsection). Then the coproduct $\Delta: C(\G) \to C(\G) \otimes \cO(\G)$ (note that we use on purpose two different notations for the same algebra, to highlight different roles it plays) defines an action of $\G$ on $C(\G)$. The fact that the action preserves $\psi$ is nothing but the invariance of the Haar state. As the coproduct action of $\G$ on $C(\G)$ is faithful, we can view $\G$ as a quantum subgroup of $\Qut(\G, \psi)$.

The following is the key definition for our paper -- for alternative versions and an extended discussion we refer to \cite[Section 9]{Daw24}. 

\begin{defn}
  Let $(\X, A)$  be a finite quantum graph. An action of a compact quantum group $\Hb$  on $(\X,A)$ is given by an action $\alpha$ of $\Hb$ on $C(\X)$ which preserves $\psi$ and the adjacency matrix $A$, where the last condition is understood via the commutation relation:
  \[ (A \otimes \id_{\cO(\Hb)}) \circ \alpha = \alpha \circ A \]
  (note that we view above $A$ as acting on $C(\X)$).
\end{defn}

The following result is now easy to deduce (see for example \cite[Proposition 9.9]{Daw24}).

\begin{prop} \label{prop:Matt}
  Let $(\X, A)$  be a finite quantum graph. Then there exists a universal compact quantum group $\Qut(\X, A)$ acting (faithfully) on $(\X,A)$ (via a map $\alpha: C(\X) \to C(\X) \otimes \cO(\Qut(\X, A))$); the universality means that for any other action $\beta:C(\X) \to C(\X) \otimes \cO(\Hb)$ of a compact quantum group $\Hb$ on $(\X,A)$ there exists a unique Hopf $^*$-algebra morphism $\gamma: \cO(\Qut(\X, A)) \to \cO(\Hb)$ such that 
  \[ \alpha = \beta \circ \gamma. \]
The \emph{quantum automorphism group} $\Qut(\X, A)$ is a quantum subgroup of $\Qut(\X)$.
\end{prop}

Note that some authors use the notation $\Aut^+(\X, A)$ instead of $\Qut(\X, A)$, and call it the \emph{quantum symmetry group} of $(\X,A)$

The last statement of the above proposition implies in particular that $\Qut(\X, A)$ is a compact matrix quantum group (as so is $\Qut(\X)$). Thus the definition above can be reinterpreted as follows (note that this approach was followed in \cite{BCEHPSW20} and in \cite{Gro22}): if $(e_i)_{i=1}^n$ is some orthogonal basis of $\ell^2(\X)$, $\cO(\Qut(\X, A))$ is the universal *-algebra generated by elements $\{u_{ij}:i,j =1, \ldots, n\}$ subject to relations
\begin{equation}
\label{eq.QAut}
\text{$U=(u_{ij})_{i,j=1}^n$ is unitary,}\quad m(U\otimes U)=Um,\quad U\eta=\eta,\quad UA=AU,
\end{equation}
where $U$ is treated as a matrix acting on $\ell^2(X)$ with entries in $\cO(\G)$ and the equalities above have to be interpreted accordingly. Diagrammatically,
\begin{equation}
\label{eq.QAutDiag}
\Diagram{\DMor{square}[-/-=] (0,1) {$U$} \DMor{square}[-/-=] (.5,-1) {$U^*$}\draw [very thick] (1,0) .. controls +(up:0.5) and +(down:0.5) .. (1.5,2); \Dmor{bcirc}[==/=] (1,2.5) \draw (-0.5,2) -- (-0.5,3);}=\Diagram{\draw (0,-2) -- (0,3); \Dmor{bcirc}[/=] (1,0.5) \draw [very thick] (1,1) -- (1,3);}=\Diagram{\DMor{square}[-/-=] (0,1) {$U^*$} \DMor{square}[-/-=] (.5,-1) {$U$}\draw [very thick] (1,0) .. controls +(up:0.5) and +(down:0.5) .. (1.5,2); \Dmor{bcirc}[==/=] (1,2.5) \draw (-0.5,2) -- (-0.5,3);},\qquad
\Diagram{\DMor{square}[-/-0=] (0,0) {$U$} \DMor{square}[-/-0=] (1,0) {$U$} \Dmor{bcirc}2/1 (-.5,1.5) \Dmor{bcirc}[==/=] (1.5,1.5)}
=\Diagram{\Dmor{bcirc}2/1 (0,-.5) \DMor{square}[-/-=] (0,1) {$U$}},\qquad
\Diagram{\Dmor{bcirc}0/1 (0,-.5) \DMor{square}[-/-=] (0,1) {$U$}}
=\Diagram{\Dmor{bcirc}0/1 (0,.5) \Dmor{bcirc}[/=] (1,.5)},\qquad
\Diagram{\DMor{square}1/1 (0,-.5) {$A$} \DMor{square}[-/-=] (0,1.5) {$U$}}
=\Diagram{\DMor{square}[-/-=] (0,-.5) {$U$} \DMor{square}1/1 (-.5,1.5) {$A$} \draw[very thick] (0.5,0.5) -- (0.5,2.5);},
\end{equation}
where thick lines denote the space $\cO(\G)$, whereas the thin lines denote the space $\ell^2(\X)$. 

Note that  the abelianisation of $\cO(\G)$ is simply the Hopf *-algebra $\cO(G)$, where $G = \Aut(\X,A)$ is the compact matrix group of 
{\it classical automorphisms} of $(\X,A)$. 
 These are precisely the *-automorphisms of $C(\X)$ that commute with $A$.
When $\Gr = (\X,A)$ is a classical finite graph, $\G = \Qut(\X,A)$ is the well-studied quantum automorphism group of the graph $\Gr$. In this case, the fundamental matrix $U = (u_{ij})_{i,j=1}^n$ is a magic unitary, and $\G$ is a quantum subgroup of the quantum permutation group $S_n^+$.  See \cite{BBC07} for a survey on these quantum groups.

It is easy to see that if $(\X,A)$ is a complete or edgeless quantum graph, then $\Qut(\X,A)= \Qut(\X)$ (the commutation relation with $A$ is then automatic). For more examples we refer to \cite{Gro22}.

\begin{remark} \label{rem:coloured}
Suppose that we are given a family $(\X,A_i)_{i \in I}$ of quantum graphs on the same finite quantum set $\X$. It is then easy to see that we can make sense of $\bigcap_{i \in I} \Qut(\X, A_i)$ (for example as a quantum subgroup of $\Qut(\X, \psi)$), and it is a universal compact quantum group acting on $C(\X)$ with the action preserving $\psi$ and each of the quantum adjacency matrices $A_i$. Classically this construction can be described as the study of symmetries of an \emph{edge-coloured} graph, with the adjacency matrix $A_i$ corresponding to a colour $i$; since we want the edges to be given a single colour, this requires an extra disjointness assumption about the $A_i$'s, namely that they satisfy $A_i\bullet A_j = \delta_{ij} A_i$. In the quantum setting we will be forced to work with more general families of quantum adjacency matrices, not necessarily ones consisting of mutually Schur-orthogonal maps. We will also comment in Subsection \ref{subsec:disjoint} on when we can work with genuine coloured quantum graphs. 
\end{remark}

Finally we mention an easy result which will be of use later. Whenever we have an action $\alpha: C(\X) \to C(\X)\otimes C(\mathbb{G})$, we get an induced map $\widetilde{\alpha}: \ell^{2}(\X) \to \ell^{2}(\X) \otimes C(\G)$, so all the usual notions for actions extend to the level of $\ell^{2}(\X)$\footnote{Since we are dealing with finite-dimensional objects, it is just a matter of identification. 
In the infinite-dimensional case, we would have to work with the Koopman representation, which is particularly nice when the action is trace-preserving.}.

\begin{prop}\label{prop:eigenvalues}
  Given a quantum graph $(\X, A)$, the eigenspaces of $D_A$ (and $D^A$) are invariant with respect to the action of $\Qut(\X, A)$. That is, the projections on the eigenspaces of $D_A$ (and $D^A$) are intertwiners of the fundamental representation $U$ of $\Qut(\X,A)$
\end{prop}
\begin{proof}
It is well known that representations and their intertwiners of compact quantum groups form a monoidal $*$-category (like for ordinary compact groups). Denote by $\Hom(U,V)$ the space of intertwiners between $U$ and $V$. By definition of $\Qut(\X,A)$, we have $m\in\Hom(U\otimes U,U)$, $\eta\in\Hom(1,U)$, $A\in\Hom(U,U)$, where $U$ is the fundamental representation of $\Qut(\X,A)$. Consequently, $D_A=m(A\eta\otimes I_{l^2(\X)})\in\Hom(U,U)$. The projections on the eigenspaces of $D_A$ are polynomials in $D_A$, so they also have to be in $\Hom(U,U)$.
\end{proof}

\section{Quantum Cayley graphs}
\label{sec.step1}

We will now briefly discuss quantum Cayley graphs, as defined and studied in \cite{Was24} (in a broader context of arbitrary, possibly infinite, discrete quantum groups).  
Let $\G$ be a finite quantum group, as introduced in the previous section. 
The Hopf $^*$-algebra $C(\G)$ is, apart from the usual algebraic multiplication, equipped also with a convolution product (essentially transported from the dual quantum group via the Fourier transform). It will be easier to work directly with convolution operators attached to functionals on $C(\G)$. Given $\omega \in C(\G)^*$ we can define 
\[ L_\omega = (\omega \otimes \id_{C(\G)}) \circ \Delta: C(\G) \to C(\G),\]
\[ R_\omega = (\id_{C(\G)} \otimes \omega) \circ \Delta: C(\G) \to C(\G).\]
Given an element $x \in C(\G)$ we define a functional $\psi_x: C(\G) \to \C$, $\psi_x(y) = \psi(xy)$ for $y \in C(\G)$, where $\psi$ is the Haar functional of $\G$, normalised to be a $1$-form. The convolution of two elements $x_1, x_2 \in C(\G)$ is defined as the density with respect to $\psi$ of the functional $\psi_{x_1}\ast \psi_{x_2}$. It can be written more explicitly as $L_{\psi_{S(x_1)}}(x_2)$, where $S$ is the antipode on $C(\G)$. We define the convolution operator $x \star: = L_{\psi_{S(x)}}$.

\begin{remark}
 In the case of finite quantum groups, $\psi$ is tracial \cite{VDa97} and one can simply describe the convolution of $x_1,x_2 \in C(\G)$ as $x_1 \star x_2=\ol{\Delta^*} (x_1 \otimes x_2)=\Delta^* (x_1 \otimes x_2)$ since it holds for arbitrary $y\in C(\G)$ that
\begin{align*}
 & \langle (x_1 \star x_2)^*, y \rangle
 = \psi_{x_1 \star x_2}(y)
 = (\psi_{x_1} \ast \psi_{x_2})(y)
 = \psi^{\otimes2}((x_1 \otimes x_2)(\Delta y))
 \\
 &= \langle x_1^* \otimes x_2^*, \Delta y \rangle
 = \langle \Delta^* (x_1^* \otimes x_2^*), y \rangle
 = \langle (\ol{\Delta^*}(x_1 \otimes x_2))^*, y\rangle
 = \langle (\Delta^*(x_1 \otimes x_2))^*, y\rangle,
\end{align*} 
where the traciality of $\psi$ ensures (by \cite[Proposition 2.15]{Mat22}) that $\Delta^*$ is real if and only if $\Delta$ is so, while the latter is one of the defining properties of $\Delta$. The other equality $\ol{\Delta^*} (x \otimes \id) = L_{\psi_{S(x)}}$ follows from
\[
\ol{\Delta^*}=(\psi \circ m \circ\Sigma \otimes \id_{C(\G)})(S^{-1} \otimes \Delta), 
\]
where $\Sigma$ is the tensor flip.
One can check this equality by precomposing the mutual inverse maps 
$(m \otimes \id)(\id \otimes \Delta)$ and 
$(m \otimes \id)(\id \otimes S \otimes \id)(\id \otimes \Delta)$ to the identity of conjugate and transposed adjoint: 
\[
\ol{\Delta^*}=(\eta_{\G^2}^* m_{\G^2} \otimes \id_\G)
(\id_{\G^2} \otimes \Delta \otimes \id_{\G})
(\id_{\G^2} \otimes m_\G^*\eta_\G),
\]
where $C(\G^2)=C(\G)^{\otimes2}$ with $m_{\G^2}=(m_\G \otimes m_\G)(\id_\G \otimes \Sigma \otimes \id_\G)$. 
Diagrammatically, we have 
\begin{equation*}
\Diagram{\DMor{square}1/2 (1,0) {$\Delta$} 
 \draw (-0.5,-1) -- (-0.5,1);
 \Dmor{bcirc}2/1 (0,1.5) 
 \draw (1.5,1) to[out=90,in=-90] (1,2); 
 \DMor{square}2/1 (.5,3) {$\ol{\Delta^*}$} ;}
=\Diagram{\Dmor{bcirc}0/2 (2,-1.5)
 \DMor{square}1/2 (0,0) {$\Delta$}
 \DMor{square}1/2 (1.5,0) {$\Delta$} 
 \draw (-.5,1)to(-.5,2)(0.5,1)to[out=90,in=-90](1,2)
 (1,1)to[out=90,in=-90](.5,2)(2,1)to(2,2);
 \Dmor{bcirc}2/1 (0,2.5) 
 \Dmor{bcirc}2/0 (1.5,2.5)
 \Dmor{bcirc}2/0 (-0.5,3.5)
 \draw (-1,-1)--(-1,3) (2.5,-1)--(2.5,3);}
=\Diagram{\Dmor{bcirc}0/2 (1.5,-1)
 \DMor{square}1/2 (0.5,1.5) {$\Delta$}
 \Dmor{bcirc}2/1 (0.5,0)
 \Dmor{bcirc}2/0 (-0.5,3)
 \Dmor{bcirc}1/0 (1,3)
 \draw (-1,-.5)--(-1,2.5) (2,-0.5)--(2,3);}
=\Diagram{\Dmor{bcirc}1/0 (0,.5)
 \draw (0,-1)--(0,0)(1,-1)--(1,2);}
=\psi \otimes \id,
\end{equation*}
hence precomposition by $(m \otimes \id)(\id \otimes S \otimes \id)(\id \otimes \Delta)$ yields 
\[
\ol{\Delta^*}=(\psi m (\id \otimes S) \otimes \id)(\id \otimes \Delta)
=(\psi S m \Sigma (S^{-1} \otimes \id) \otimes \id)(\id \otimes \Delta)
=(\psi m \Sigma \otimes \id)(S^{-1} \otimes \Delta).
\]
Since $\psi$ is tracial, i.e., $\G$ is of Kac type, $S^2=\id$ and $\psi m\Sigma=\psi m$ holds, thus
$\ol{\Delta^*}(x\otimes y)=(\psi m \otimes \id)(S(x) \otimes \Delta(y))=L_{\psi_{S(x)}}y$.
\end{remark}

\begin{defn}\cite{Was24} \label{def:Qcayley}
    Let $\G$ be a finite quantum group. A quantum Cayley graph on $\G$ is given by the adjacency matrix of the form $Ax = P \star x$ ($x \in C(\G)$), where $P \in C(\G)$ is an orthogonal projection.
\end{defn}

The fact that $P$ being a self-adjoint projection is equivalent to the convolution operator given by $P$ yielding the adjacency matrix of a quantum graph on $\G$ is proved in \cite[Section 4.2]{Was24}. Note that \cite{Was24} often imposes also additional requirements that $\epsilon(P) = 0$ (equivalent to the quantum graph having no loops) and $P= S(P)$ (equivalent to the quantum graph being undirected). We will however need to consider also directed quantum graphs.

The analogy with the classical case is stronger if we in addition require that the projection $P$ defining our quantum Cayley graph is \emph{central}. Note that central projections in $C(\G)$ are in one-to-one correspondence with subsets of $\textup{Irr} (\hat{\G})$, the set of equivalence classes of irreducible representations of the dual quantum group. We also have a natural notion of such a subset being \emph{generating}, which we will mention after the next lemma; for more details and references we refer to \cite[Section 4]{Was24}, see also \cite{Fim10}.

The following easy observation is essentially a sanity check: given a quantum Cayley graph $\mathcal{G}$ associated to a finite quantum group $\G$ we expect that $\G$ would be acting on $\mathcal{G}$.

\begin{lem} \label{lem:G as a subgroup}
Let $\G$ be a finite quantum group and let $(\G, A)$ be a quantum Cayley graph. Then $\G$ is a quantum subgroup of $\Qut(\G,A)$: there exists a surjective Hopf $^*$-morphism $q:\cO(\Qut(\G,A)) \to \cO(\G)$ such that 
\[  (\id_{\cO(\G)} \otimes q) \circ \alpha_{\Qut(\X,A)} = \Delta_\G.  \]
\end{lem}
\begin{proof}
Let $\omega \in C(\G)^*$. By the universal property of $\Qut(\G,A)$ and faithfulness of the action given by $\Delta_\G$ it suffices to check that the left convolution operator $L_\omega$ is preserved by the action of $\G$ given by the coproduct, i.e.\ that
 \[ (L_\omega \otimes \id_{\cO(\G)}) \circ \Delta_\G = \Delta_\G \circ L_\omega. \]
This is however an immediate consequence of the coassociativity of $\Delta_\G$.     
\end{proof}

We finish this section by relating the quantum Cayley graphs as described above to the quantum Cayley graphs appearing much earlier in work of Vergnioux \cite{Ver05}. Let $\G$ be a finite quantum group. The definition of Vergnioux takes as a starting point a finite set $S \subset \textup{Irr}(\widehat{\G})$ (assumed to be {\it symmetric} and \emph{generating} in the sense that $S = \bar S$ and all $\alpha \in \Irr(\widehat\G)$ appear as sub-objects of tensor products of elements of $S$).  As above, we consider the associated central orthogonal projection $P_S \in C(\G)$. Then \cite{Ver05} associates to $P_S$ a triple $(K,\mc S, \mc T)$, 
where $K = P_S C(\G) \otimes C(\G)$ is a quantum counterpart of the set of functions on the Cartesian product $S \times G$, and $\mc S, \mc T: C(\G) \to K$ are the {\it source} and {\it target} maps, respectively, given by 
\[
\mc S f = P_S \otimes f,  \qquad \mc T f = (P_S \otimes 1)\Delta(f) \qquad (f \in C(\G)).  
\]
In the classical case, these maps correspond to usual source and target maps $s,t:S \times G \to G$ via Gelfand duality.  Namely, \[(\mc S f)(s,g) = f(s(s,g)) = f(g)\] and \[\mc (Tf)(s,g) = f(t(s,g)) = f(sg), \;\;\; s \in S, g \in G,\]
The space $K$ comes equipped with a natural $C(G)$-bimodule structure given by 
\[
f \cdot (P_Sa \otimes  b)\cdot h := \Delta(f)(P_Sa \otimes b)(1 \otimes h) \qquad (f,h,a,b \in C(\G)).
\]
Classically, this corresponds to the actions
\[
(f\cdot \varphi \cdot h)(s,g) = f(sg)\varphi(s,g)h(g) \qquad (f,h \in C(G), \varphi \in K = C(G \times S), s \in S, g \in G). 
\]
We can promote $K$ to a C*-correspondence over $C(\G)$ if we equip it with the  $C(\G)$-valued innner product given by
\[\langle\xi| \eta \rangle_{C(\G)} = (\psi \otimes \id) (\xi^*\eta) \qquad (\xi,\eta \in K). \]

We now show that the above Vergnioux quantum Cayley graphs are in one-to-one correspondence with the quantum Cayley graphs $(\G, A)$, where $A= P_S \star (\cdot)$.  This is done by making use of the {\it edge correspondence}  introduced in \cite{BHINW23}. Given a quantum graph $(\X,A)$, one can associate the {\it quantum edge correspondence} $E_A = C(\X) \otimes_A C(\X)$\footnote{Actually, we are defining here a C*-correspondence which is isomorphic to the edge correspondence.  See \cite[Corollary 2.6]{BHINW23}}, which is the separation and completion of the standard $C(\X)$-bimodule $C(\X) \otimes C(\X)$ with respect to the $C(\X)$-valued inner product 
\[
\langle a \otimes b, c \otimes d\rangle_{C(\X)} = b^*A(a^*c)d, \;\; a,b,c,d \in C(\X).
\]
Note that from \cite[Theorem 2.5 and Corollary 2.6]{BHINW23}, it follows that the quantum edge correspondences $E_A$ and quantum graphs $(\X,A)$ are in one-to-one correspondence. The following result shows that the Vergnioux correspondence $K$ and $E_A$ are canonically isomorphic.    
\begin{thm}\label{correspondence-isomorphism}
Let $\G$ be a finite quantum group, and let $S\subset  \textup{Irr}(\widehat{\G})$ be a symmetric generating set. Let $(C(\G), A)$ be the quantum Cayley graph associated to $P_S$, let $E_A$ be its quantum edge correspondence and let $K$ be the $\Cst$-correspondence in the Vergnioux triple $(K,\mc S, \mc T)$ determined by $P_S$.
The linear map $\Phi:E_A \to K$ given by a naturally interpreted formula
\[
\Phi(a\otimes b) = \Delta(a)(1 \otimes b)(P_S \otimes 1), \;\; a, b \in C(\G),
\]
is an isomorphism of $\Cst$-correspondences. Hence finite quantum Cayley graphs of \cite{Ver05} are in 1-1 correspondence to finite central undirected quantum Cayley graphs of \cite{Was24}.
\end{thm}

\begin{proof}
Since $C(\G) \otimes C(\G)$ is spanned by elements of the form $\Delta(a)(1 \otimes b)$, the map $\Phi$ will be surjective, provided it is well defined.  In that case, it is also clear that $\Phi$ will be a $C(\G)$-bimodule map. We now show that $\Phi$ is in fact a well-defined isometry by showing that it preserves the $C(\G)$-valued inner products.  Let $\xi = \sum_i a_i \otimes b_i \in C(\G) \otimes C(\G)$. We then have 
\begin{align*}
\langle \Phi(\xi), \Phi(\xi) \rangle_{C(\G)} &= \sum_{i,j} (\psi \otimes \id) (P_S \otimes b_i^*)\Delta(a_i^*a_j)(P_S \otimes b_j) = \sum_{i,j} b_i^* (P_S\star (a_i^*a_j)) b_j \\
&=\sum_{i,j} b_i^* A(a_i^*a_j) b_j = \langle \xi | \xi \rangle_{C(\G)}. 
\end{align*}
\end{proof}
\begin{remark}
The Vergnioux correspondence classically corresponds to edges identified as pairs $(s,g)$, while the edge correspondence uses the pairs $(g,sg)$, which motivated the formula we used for the isomorphism. The equivalence above persists also in the case of quantum Cayley graphs of general discrete quantum groups, as considered in \cite{Was24}.
\end{remark}

\section{A `Coloured' version of Frucht's theorem} \label{sec:coloured-Frucht}

In this section, we quantize the first part of the proof of the Frucht's theorem, which classically corresponds to the following statement.

\begin{prop}
\label{P.classical1}
    Let $G$ be a finite group. For $g\in G$, denote by $\Gr_g$ the directed Cayley graph of $G$ with respect to the one-element set $S=\{g\}$. Then $G=\bigcap_{g\in G}\Aut\Gr_g$.
\end{prop}
\begin{proof}
    In each $\Gr_g$, there is an edge $h\to k$ if and only if $k=gh$. Let $\phi$ be a common automorphism of all the graphs. If $\phi$ maps $e\mapsto h$, then it has to map $g\mapsto gh$  since in $G_g$, $g$ is the unique neighbour of $e$ and $gh$ is the unique neighbour of $h$. That is, $\phi$ is the right multiplication by $h$. Conversely, the right multiplication by $h$ clearly is an automorphism of each $\Gr_g$ for every $h$.
\end{proof}

Before we pass to its quantum counterpart, we prove an easy algebraic lemma, which can be informally phrased as follows: if an action on a given (quantum) group commutes with all the left translation operators, it must be given by right translations. We will formulate it in the context of actions on finite quantum groups, as this is all we need and allows us to avoid any technicalities, reducing the problem to an algebraic computation.

\begin{lem}\label{lem:commute}
Let $\G$ be a finite compact quantum group. Suppose that $\alpha: C(\G) \to C(\G) \otimes \cO(\Hb)$ is an action of a compact quantum group $\Hb$ on $C(\G)$ such that for every functional $\omega\in \C(\G)^*$  we have
\begin{equation} ( L_\omega  \otimes \id_{\cO(\Hb)}) \circ \alpha = \alpha \circ L_\omega. \label{commuteleft}\end{equation}
Then there exists a (unique) Hopf $^*$-morphism $\gamma: \cO(\G) \to \cO(\Hb)$ intertwining the actions, i.e.\ such that 
\begin{equation} \label{intertwinactions}  (\id_{\cO(\G)} \otimes \gamma) \circ \Delta_\G = \alpha . \end{equation}
\end{lem}

\begin{proof}
Set $\gamma = (\epsilon_\G \otimes \id_{\cO(\Hb)}) \circ \alpha:\cO(\G) \to \cO(\Hb)$. It is clear that $\gamma$ is a unital $^*$-homomorphism.
Note that the equation \eqref{commuteleft} can be rewritten as
\[ (\omega \otimes \id_{\cO(\G)} \otimes \id_{\cO(\Hb)}) \circ (\Delta_\G \otimes \id_{\cO(\Hb)}) \circ \alpha = 
(\omega \otimes \id_{\cO(\G)} \otimes \id_{\cO(\Hb)}) \circ (\id_{\cO(\G)} \otimes \alpha) \circ \Delta_\G,\]
so, as $\omega$ is arbitrary, is equivalent to 
\begin{equation} \label{commuteleftglobal}
(\Delta_\G \otimes \id_{\cO(\Hb)}) \circ \alpha = 
 (\id_{\cO(\G)} \otimes \alpha) \circ \Delta_\G.
\end{equation}
This implies in turn that 
\begin{align*} (\id_{\cO(\G)} \otimes \gamma) \circ \Delta_\G  &=
(\id_{\cO(\G)} \otimes \epsilon_\G \otimes \id_{\cO(\Hb)}) (\id_{\cO(\G)}\otimes \alpha) \circ \Delta_\G \\&= (\id_{\cO(\G)} \otimes \epsilon_\G \otimes \id_{\cO(\Hb)}) \circ (\Delta_\G \otimes \id_{\cO(\Hb)}) \circ \alpha = \alpha.
\end{align*}
so that \eqref{intertwinactions} holds.

It remains to check that $\gamma$ intertwines the respective coproducts. Indeed,
\begin{align*} 
(\gamma \otimes \gamma) \circ \Delta_\G &= 
(\gamma \otimes \epsilon_\G \otimes \id_{\cO(\Hb)}) \circ (\id_{\cO(\G)} \otimes \alpha) \circ \Delta_\G \\&= (\gamma \otimes \epsilon_\G \otimes \id_{\cO(\Hb)}) 
(\Delta_\G \otimes \id_{\cO(\Hb)}) \circ \alpha = (\gamma \otimes \id_{\cO(\Hb)}) \circ \alpha \\&= 
(\epsilon_\G \otimes \id_{\cO(\Hb)} \otimes \id_{\cO(\Hb)}) \circ (\alpha \otimes \id_{\cO(\Hb)}) \circ \alpha \\&= (\epsilon_\G \otimes \id_{\cO(\Hb)} \otimes \id_{\cO(\Hb)}) \circ (\id_{\cO(\G)} \otimes \Delta_{\Hb}) \circ \alpha =
\Delta_{\Hb} \circ \gamma.
\end{align*}

\end{proof}

In the following proposition we formulate and prove a quantum version of Proposition \ref{P.classical1}. Observe that by spectral decomposition, any finite dimensional C*-algebra is spanned by its orthogonal projections.

\begin{prop}
\label{P.quantum1}
    Let $\G$ be a finite quantum group. Then there exists a finite collection $\{P_i\}_{i=1}^k$ of orthogonal projections in $C(\G)$ such that  the corresponding quantum Cayley graphs $(\G, A_i)$ with $A_i = P_i\, \star$ for each $i=1, \ldots, k$ have no loops and $\G \cong \bigcap_{i=1}^k  \Qut(\G, A_i)$. 
\end{prop}
\begin{proof}
Recall that a quantum Cayley graph associated with a projection $P \in C(\G)$ does not have loops if and only if $\epsilon(P)=0$. The structure of algebras associated with finite quantum groups guarantees that $C(\G) = \C P_\epsilon \oplus D$, where $P_\epsilon$ is the support projection of the counit and $D$ is another finite-dimensional $\Cst$-algebra. Choose a collection $\{P_i\}_{i=1}^k$ of orthogonal projections spanning $D$ and set $A_i = P_i\, \star$ for each $i=1, \ldots,k$. By the above argument $(C(\G),A_i)$ are quantum Cayley graphs without loops.

Lemma \ref{lem:G as a subgroup} implies that $\G$ is a quantum subgroup of $\G':=  \bigcap_{i=1}^k  \Qut(\G, A_i)$, i.e.\ we have a Hopf $^*$-morphism  $q: \cO(\G') \to \cO(\G)$ intertwining the respective actions. Consider then the action $\alpha'$ of $\G'$ on $\C(\G)$. By the assumption (see Remark \ref{rem:coloured}) it commutes with all the operators $P_i \, \star = L_{\psi_{S(P_i)}}$ (recall the notation introduced before Definition \ref{def:Qcayley}), and of course also with $L_{\varepsilon} = \id_{C(\G)}$. As the map $x \mapsto \psi_x$ is a linear bijection from $\C(\G)$ onto $C(\G)^*$, the action $\alpha'$ satisfies the assumptions of Lemma \ref{lem:commute}. This gives us a Hopf $^*$-morphism $\gamma: \cO(\G) \to \cO(\G')$, again intertwining the actions. 

\end{proof}

\begin{remark}
    It is worth noticing that the projections involved in the statement in general cannot be mutually orthogonal as elements of $\ell^2(G)$, as that would imply that  $C(\G)$ is commutative, and $\G$ would necessarily be a classical group. This means that the corresponding quantum Cayley graphs are not ``mutually disjoint'', that is, they do not satisfy $A_{i}\bullet A_{j}=A_{j}\bullet A_{i}=\delta_{ij} A_{i}$ for $i, j=1, \ldots, k$ (see Remark \ref{rem:coloured}), so they do not define a coloured quantum graph. 
    Later in this section we will see that in many cases one can prove the coloured version of the quantum Frucht's theorem with disjoint colours. Nevertheless, it is not needed for our proof of the general, uncoloured version of the quantum Frucht's theorem. 
\end{remark}

\begin{remark}
    Allowing the set of colours to be a quantum set, one can think of $(\G,\{A_i\})$ in Proposition \ref{P.quantum1} as a complete quantum Cayley graph disjointly coloured by the quantum set $\G$. 
    Note that a disjoint $k$-colouring of the edges of $(\X,A)$ is nothing but a $^*$-homomorphism from $\C^k$ to the Schur product algebra $(B(\ell^2(\X)),\bullet,\ol{(\cdot)})$ mapping $1$ to the adjacency matrix $A$.
    Similarly, given a quantum set $\mathbbm{k}$ of `colours', an edge $\mathbbm{k}$-colouring of $(\X,A)$ can be defined as a $^*$-homomorphism from $C(\mathbbm{k})$ to the Schur product algebra $B(\ell^2(\X))$ mapping $1_\mathbbm{k}$ to $A$.

    In this sense, the convolution homomorphism $C(\G) \ni x \mapsto x\star \in B(\ell^2(\G))$ is an edge $\G$-colouring of the complete graph $\eta \psi$ with loops, and the quantum automorphisms preserving (commuting with) `colours' $x\star$ are exactly the quantum automorphisms commuting with the collection $\{A_i\}_{i=1}^k$.
\end{remark}

Finally we record one more consequence of Lemma \ref{lem:commute} and the proof of Proposition \ref{P.quantum1}.

\begin{prop} \label{prop:fullalgebra}
   Suppose that $\G$ is a finite quantum group and $P \in C(\G)$ is an orthogonal projection such that its convolution powers generate $C(\G)$ (in other words, the Fourier transform $\mathcal{F}(P)$ generates $C(\widehat{\G})$ as an algebra). Set $A=P \, \star$. Then we have  
   $\G = \Qut(\G, A)$.
\end{prop}
\begin{proof}
    Follows from Lemmas \ref{lem:G as a subgroup} and \ref{lem:commute} exactly as in the proof of Proposition \ref{P.quantum1}. We leave the details to the reader.
\end{proof}

This naturally raises a following question.

\begin{question}
   Which finite quantum groups $\G$ admit an orthogonal projection $P \in C(\G)$ whose convolution powers generate $C(\G)$?
\end{question}

We will provide some partial answers to this question in Section \ref{sec:rigid-calyey}.

\subsection{Disjoint colours}\label{subsec:disjoint}
\begin{defn}
Let $\X$ be a quantum space and let $(A_i)_{i=1}^k$ be a finite family of quantum adjacency matrices. We say that this family defines a \emph{coloured quantum graph} if we have $A_i \bullet A_j = A_j\bullet A_i = \delta_{ij} A_i$ for $i,j=1, \ldots, k$.
\end{defn}
We will now investigate when we can impose a coloured quantum Cayley graph structure on $C(\G)$ such that the quantum automorphism group is equal to $\G$. In this case the $A_i$'s are convolution operators by some projections $P_i \in C(\G)$ and the condition $A_i \bullet A_j = A_j\bullet A_i = \delta_{ij} A_i$ is equivalent to saying that the corresponding projections are mutually orthogonal. Therefore we obtain a maximal family if we take projections $P_i$ to be rank one, i.e.\ in each matrix block of $C(\G)$ we take all the diagonal matrix units. The only freedom we have is choosing an orthonormal basis for each of the blocks. For the quantum automorphism group of the resulting coloured quantum graph the algebra of intertwiners is also closed under the Schur product. This means that we would like to prove that the subalgebra of $C(\G)$ generated by the projections $P_i$, closed under both the usual and the convolution products, is equal to $C(\G)$. On the dual side, it means that the subalgebra of $C(\widehat{\G})$ generated by the diagonal matrix coefficients, closed under both the convolution and the usual multiplication, is equal to $C(\widehat{\G})$. Because of this goal, we will need a formula for the product of two matrix coefficients.

\begin{prop}[{\cite{KS18}, \cite[Lemma 4.1]{KS17}}]\label{prop:coeffmultiplication}
Let $\widehat{\G}$ be a compact quantum group and let $\op{Irr}(\widehat{\G})$ be the set of its irreducible representations. For any $\beta,\gamma \in \op{Irr}(\widehat{\G})$, any $b,b' \in \{1,\dots, n_{\beta}\}$ and $c,c'\in \{1,\dots, n_{\gamma}\}$ we have
\begin{equation}\label{Eq:coeffmultiplication}
u_{b,b'}^{\beta} u_{c,c'}^{\gamma} = \sum_{\alpha \in \op{Irr}(\G)} \sum_{i=1}^{m(\alpha, \beta\otimes \gamma)} \sum_{a,a'} V(\alpha, \beta\otimes \gamma, i)_{a}^{b,c} u_{a,a'}^{\alpha} \overline{V(\alpha,\beta\otimes \gamma,i)_{a'}^{b',c'}},
\end{equation}
where $V(\alpha,\beta\otimes\gamma, i)$ is an isometric intertwiner from $\alpha$ to the $i$-th copy of $\alpha$ inside $\beta\otimes \gamma$.
\end{prop}
\begin{remark}\label{rem:cuttingcoeff}
Because we can not only form the usual products but also convolutions, whenever a matrix coefficient appears in a sum with a non-zero scalar, we know that it belongs to our algebra. It follows from the fact that $u_{a,a}^{\alpha} \star u_{d,d'}^{\alpha} \star u_{a',a'}^{\alpha} = \kappa \delta_{ad} \delta_{a'd'} u_{a,a'}^{\alpha}$ for some  $\kappa>0$, so using only diagonal coefficients we can extract any coefficient from a linear combination using convolutions.
\end{remark}
We will work with randomly selected orthonormal bases and we need to clarify what we mean by saying that something happens for almost all choices. The set of orthonormal bases is in one-to-one correspondence to the set of unitary matrices, which form an irreducible, real variety, but also a compact group. If we have a few Hilbert spaces and choose an orthonormal basis in each of them, then the relevant space will be a product of unitary groups, which is also an irreducible variety and a compact group. It means that we can talk about the Zariski topology and the Haar probability measure. Because of irreducibility, any non-empty Zariski open set will be of full measure. We will state the results in probabilistic terms, but in the proofs we will use the Zariski topology and the connection between openness and being of full measure. We will start with a simple lemma.

\begin{lem}\label{lem:nonzerocoeff} Let $\G$ be a finite quantum group and 
let $\alpha, \beta, \gamma$ be different irreducible representations of $\widehat{\G}$. If $V: \mathsf{H}_{\alpha} \to \mathsf{H}_{\beta} \otimes \mathsf{H}_{\gamma}$ is a non-zero operator (e.g.\ a non-zero intertwiner between $\alpha$ and $\beta \otimes \gamma$) then for almost all choices of orthonormal bases of $\mathsf{H}_{\alpha}, \mathsf{H}_{\beta}$, and $\mathsf{H}_{\gamma}$ all matrix coefficients $\langle f^{\beta}_{b}\otimes f^{\gamma}_{c}, V f^{\alpha}_{a}\rangle$ are non-zero.
\end{lem}
\begin{proof}
Ultimately we will have to choose an orthonormal basis in each carrier space of an (equivalence class of an) irreducible representation. In this particular lemma we want to have the freedom to choose the bases in the three Hilbert space independently, so they have to come from distinct irreducible representations.

Any orthonormal basis can be written using a fixed orthonormal basis and a unitary matrix, therefore the matrix coefficients can be written as
\[
\langle U_{\beta}e^{\beta}_{b} \otimes U_{\gamma} e^{\gamma}_{c}, VU_{\alpha} e^{\alpha}_{a}\rangle.
\]
We want to show that for some choice of these unitaries, all the inner products are non-zero; this is a Zariski open set, so of full measure if non-empty. Suppose then that for all choices of unitaries, there exists a choice of indices $a,b,c$ such that the matrix coefficient is equal to zero. But there are only finitely many choices of indices, so there must be one, say $a_0, b_0, c_0$ that works for a set of unitaries of positive measure. The condition that this particular coefficient vanishes is algebraic, so if it holds on a set of positive measure, it holds everywhere. It follows that for all choices of unitaries $U_\beta, U_\gamma, U_\alpha$  we have
\[
\langle U_{\beta}e^{\beta}_{b_0} \otimes U_{\gamma} e^{\gamma}_{c_0}, VU_{\alpha} e^{\alpha}_{a_0}\rangle = 0.
\]
Since we can choose arbitrary unitaries, the vectors $U_{\alpha}e^{\alpha}_{a_0}, U_{\beta} e^{\beta}_{b_0}, U_{\gamma} e^{\gamma}_{c_0}$ are arbitrary unit vectors in respective Hilbert spaces, i.e. $\langle \xi \otimes \eta, V\zeta\rangle =0$ for all vectors. As the simple tensors span the whole tensor product, we get $V\zeta = 0$ for all $\zeta \in \mathsf{H}_{\gamma}$, thus $V=0$, which gives a contradiction.
\end{proof}
We will now apply this lemma to quantum groups that admit a non-trivial character.
\begin{prop} Let $\G$ be a finite quantum group and 
suppose that $\widehat{\G}$ admits a non-trivial character. Then for almost all choices of orthonormal bases in the carrier spaces of irreducible representations of $\widehat{\G}$ the algebra generated by the diagonal matrix coefficients, closed under convolution, is equal to $C(\widehat{\G})$.
\end{prop}
\begin{proof}
Let $\chi$ be a non-trivial character. Let $\alpha$ be any irreducible representation of dimension greater than $1$. We either have $\alpha \otimes \chi \simeq \alpha$ or $\alpha \otimes \chi \simeq \beta$, where $\beta$ is different from $\alpha$. In the second case we have $\alpha \simeq \beta \otimes \overline{\chi}$. By formula \eqref{Eq:coeffmultiplication} the scalar coefficient of $u_{a,a'}^{\alpha}$ in the product $u_{b,b}^{\beta} u_{c,c}^{\overline{\chi}}$ is equal to $V(\alpha,\beta\otimes \overline{\chi})_{a}^{b,c} \overline{V(\alpha,\beta\otimes \overline{\chi})_{a'}^{b,c}}$; there is no sum, because the multiplicity is equal to $1$, as $\chi$ is one-dimensional. By Lemma \ref{lem:nonzerocoeff} for almost all choices these coefficients are non-zero, so $u_{a,a'}^{\alpha}$ will belong to the algebra we are interested in. 

In the case $\alpha \otimes \chi \simeq \alpha$, remembering that the carrier space of $\chi$ is $\C$, we get a map $V: \mathsf{H}_{\alpha} \to \mathsf{H}_{\alpha}$, which is not a multiple of the identity; the identity would not be an intertwiner, as $\chi$ is assumed to be non-trivial. It follows that for almost all choices of an orthonormal basis in $\mathsf{H}_{\alpha}$ all the matrix entries of $V$ are non-zero, and we can conclude just like in the first part of the proof.

To finish the proof, we can repeat the same argument for all irreducible representations of $\widehat{\G}$. In the end we obtain a finite number of Zariski open dense subsets, whose intersection is of full measure.
\end{proof}

From now on we can assume that $\widehat{\G}$ does not admit a non-trivial character. The next result will be an extension of Lemma \ref{lem:nonzerocoeff}.
\begin{prop}
Let $\widehat{\G}$ be a finite quantum group. Assume that for any non-trivial $\alpha \in \op{Irr}(\widehat{\G})$ there exist non-trivial irreducible representations $\beta$ and $\gamma$, distinct from $\alpha$ (and from each other), such that $\alpha \subset \beta\otimes \gamma$. Then for almost all choices of orthonormal bases the subalgebra generated by diagonal matrix coefficients, closed under the usual multiplication and convolution, is equal to $C(\widehat{\G})$.
\end{prop}
\begin{proof}
By Proposition \ref{prop:coeffmultiplication}, $u_{a,a'}^{\alpha}$ appears in the product $u_{bb}^{\beta} u_{cc}^{\gamma}$ with the coefficient $\sum_{i=1}^{m} \langle e_{b}^{\beta}\otimes e_{c}^{\gamma}, V_i e_{a}^{\alpha}\rangle\langle V_i e_{a'}^{\alpha}, e_{b}^{\beta}\otimes e_{c}^{\gamma}\rangle$. If this number is non-zero, then by Remark \ref{rem:cuttingcoeff} $u_{a,a'}^{\alpha}$ belongs to the algebra generated by the diagonal matrix coefficients, using both the usual products and the convolutions. If it is not true that this number is non-zero for almost all choices of orthonormal bases, then  just like in the proof of Lemma \ref{lem:nonzerocoeff} we conclude that for all vectors $\zeta, \zeta' \in \mathsf{H}_{\alpha}$, $\xi \in \mathsf{H}_{\beta}$, and $\eta \in \mathsf{H}_{\gamma}$, with $\zeta \perp \zeta'$ we have 
\[
\sum_{i=1}^{m} \langle \xi\otimes \eta, V_i \zeta\rangle\langle V_i \zeta', \xi\otimes \eta\rangle = 0.
\]
By the polarization formula, first applied to $\xi$, and then to $\eta$, we conclude that the operator $T:=\sum_{i=1}^{m} |V_i \zeta\rangle\langle V_i \zeta'|$ is equal to $0$. This is not possible because, $T V_{1}\zeta' = V_1 \zeta \|V_1 \zeta'\|^2 \neq 0$, as the ranges of the $V_i$'s are mutually orthogonal. It follows that for almost all choices we get all matrix coefficients of $\alpha$ in our algebra. Now we do this for all the irreducible representations of $\widehat{\G}$ and get a finite intersection of sets of full measure.
\end{proof}
It seems likely that the result holds for all finite quantum groups with no extra assumptions but we are unable to prove it at the moment.

\section{From multiple graphs to one graph}
\label{sec.step2}

The second part of the proof of the Frucht's theorem amounts to replacing the set of directed (quantum) graphs by a single -- possibly larger -- (quantum) graph, whose (quantum) automorphism group coincides with the intersection of the (quantum) automorphism groups of graphs in the mentioned set. 

Frucht in his original article \cite{Fru39} achieves this by replacing each directed edge $g\to h$ in the $i$-th graph by the following construction in the new graph
$$
\begin{tikzpicture}[baseline={(0,-.5ex)},every node/.style={circle,draw=white,line width=2pt,fill=black,inner sep=1.8pt,label distance=-5pt}]
    \draw (0,0) node[label=below:$g$] {} --
          (1,0) node[label=below:$a_0$] (a) {} --
          (2,0) node[label=below:$b_0$] (b) {} --
          (3,0) node[label=below:$h$] {};
    \node[label=left:$a_1$] (a1) at (1,1) {} edge (a);
    \node[label=left:$a_2$] (a2) at (1,2) {} edge (a1);
    \node[label=left:$a_{2i-1}$] (a3) at (1,3) {} edge[dotted] (a2);
    \node[label=right:$b_1$] (b1) at (2,1) {} edge (b);
    \node[label=right:$b_2$] (b2) at (2,2) {} edge (b1);
    \node[label=right:$b_{2i-1}$] (b3) at (2,3) {} edge[dotted] (b2);
    \node[label=right:$b_{2i}$] (b4) at (2,4) {} edge (b3); 
\end{tikzpicture}
$$

However, in case of quantum graphs, it seems unclear how to replace ``quantum edges'' in a useful way. On the other hand, replacing vertices seems to be much easier, as we will show below. As a result, we formulate an alternative approach to the proof of Frucht's theorem. Note that in contrast with Frucht's original construction, our approach works only for regular (quantum) graphs.  This apparent restriction, however, turns out to be fine as (quantum) Cayley graphs are always regular. To make the construction clear, we first formulate it classically and then in the quantum setting. Also we chose to first formulate a directed version of the statement, which is slightly easier to present.

\subsection{Directed version}

\begin{prop}
\label{P.classical21}
  Let $ N, n \in \N$ and let $\{\Gr_i\}_{i=1}^n$ be a set of directed regular graphs without loops on $N$ vertices. Then there exists a directed graph $\Gr$ such that $\Aut \Gr=\bigcap_{i=1}^n\Aut\Gr_i$.
\end{prop}
\begin{proof}
Denote by $X$ the vertex set of the graphs $\Gr_i$. Denote by $d_i$ the degree of regularity of each $\Gr_i$. Assume that $d_1\le\cdots\le d_n$.

We construct a (classical) simple graph $\Hr$ on $n+1$ vertices indexed by $\{0,1,\dots,n\}$ as follows. The vertex 0 is connected to all $i>\lfloor n/2\rfloor$. A pair of nonzero vertices $i\neq j$ is connected if and only if $i+j>n$. This graph has the property that each vertex $1\le i\le n$ has degree $d_i=i$. See Fig.~\ref{fig:H}.

\begin{figure}[t]
\hbox to\hsize{\hfil
\vbox{\hsize=.3\hsize\noindent\hfil
\begin{tikzpicture}[baseline={(0,-.5ex)},every node/.style={circle,draw=white,line width=2pt,fill=black,inner sep=1.8pt,label distance=-5pt}]
\draw (0,2) node [label=above:0] (0) {};
\draw ({225}:1) node[label=below:1] (1) {};
\draw ({315}:1) node[label=below:2] (2) {};
\draw ({45}:1) node[label=right:3] (3) {};
\draw ({135}:1) node[label=left:4] (4) {};

\draw (1)--(4)--(0)--(3)--(4)--(2)--(3);
\end{tikzpicture}
}
\vbox{\hsize=.3\hsize\noindent\hfil
\begin{tikzpicture}[baseline={(0,-.5ex)},every node/.style={circle,draw=white,line width=2pt,fill=black,inner sep=1.8pt,label distance=-5pt}]
\draw (0,2) node [label=above:0] (0) {};
\draw ({234}:1) node[label=below:1] (1) {};
\draw ({306}:1) node[label=below:2] (2) {};
\draw ({18}:1) node[label=right:3] (3) {};
\draw ({90}:1) node[label=above left:4] (4) {};
\draw ({162}:1) node[label=left:5] (5) {};

\draw (1)--(5)--(0)--(3)--(4)--(5)--(2)--(4)--(0);
\draw (3)--(5);
\end{tikzpicture}
}
\vbox{\hsize=.3\hsize\noindent\hfil
\begin{tikzpicture}[baseline={(0,-.5ex)},every node/.style={circle,draw=white,line width=2pt,fill=black,inner sep=1.8pt,label distance=-5pt}]
\draw (0,2) node [label=above:0] (0) {};
\draw ({210}:1) node[label=left:1] (1) {};
\draw ({270}:1) node[label=below:2] (2) {};
\draw ({330}:1) node[label=right:3] (3) {};
\draw ({30}:1) node[label=right:4] (4) {};
\draw ({90}:1) node[label=above left:5] (5) {};
\draw ({150}:1) node[label=left:6] (6) {};

\draw (1)--(6)--(0)--(4)--(5)--(6)--(2)--(5)--(3)--(6);
\draw (5)--(0);
\draw (3)--(4)--(6);
\end{tikzpicture}
}\hfil}
    \caption{Graph $\Hr$ for $n=4,5,6$.}
    \label{fig:H}
\end{figure}

Now the graph $\Gr$ is constructed by replacing each vertex $v\in X$ with a copy of $\Hr$. For $v\neq w \in X$ and $i\neq 0$, we construct a directed edge $(v,i)\to(w,i)$ whenever there is an edge $v\to w$ in $\Gr_i$.  We also construct a directed edge $(v,i)\to(v,j)$ whenever $i$ is connected to $j$ in $\mathcal H$.  Observe that any common automorphism $g$ of all the graphs $\Gr_i$ acts on $\Gr$ by $(v,i)\mapsto (g(v),i)$, so we have a natural injective embedding of  $\bigcap_{i=1}^n\Aut\Gr_i$ into $\Aut \Gr$. For simplicity we will simply write $\Aut\Gr\supseteq\bigcap_{i=1}^n\Aut\Gr_i$. Now we need to show the converse inclusion.

In $\Gr$, it holds that for $i\neq 0$ the degree of each $(v,i)$ equals to $i+d_i$. For $i=0$, the degree equals to $\lceil n/2\rceil$. So, the degree of $(v,i)$ is precisely determined by $i$. Conversely, the degrees of $(v,i)$ are distinct for different $i$ except it could happen that the degree of $(v,0)$ equals the degree of $(v,i_0)$ for some $i_0<\lceil n/2\rceil$. But we can distinguish vertices labeled by $0$ from those labelled by $i_0$ by noticing that vertices labelled by 0 are connected to vertices labelled by $\lfloor n/2\rfloor+1$, while vertices labelled by $i_0$ are not. Therefore, every automorphism of $\Gr$ preserves $i$ when acting on $(v,i)$ and hence $\Aut \Gr\subset S_X\times\Aut \Gr_1\times\cdots\times\Aut \Gr_n$. Finally, taking some $(g_0,g_1,\dots,g_n)\in\Aut \Gr$, we must have $g_0=g_1=\cdots=g_n$. Indeed, if $g_n(v,n)=(v',n)$, then $g_i(v,i)=(v',i)$ since $(v,i)$ is the unique neighbour of $(v,n)$ among vertices labelled by $i$. Hence, $\Aut \Gr\subseteq\bigcap_{i=1}^n\Aut \Gr_i$, which is what we wanted to show. 
\end{proof}

One should note that the procedure described in the above statement is far more efficient than the standard one presented in \cite{Fru39}, but requires the input graphs to be regular.

We are now ready for the quantum version.

\begin{thm} \label{thm:directed-frucht}
    Let $\X$ be a finite quantum set, let $n \in \N$ and let $\{\Gr_i\}_{i=1}^n$, $\Gr_i=(\X,A_i)$ be a set of directed regular quantum graphs without loops on $\X$. Then there exists a directed quantum graph $\Gr=(\X',A')$ such that $\Qut(\Gr)=\bigcap_{i=1}^n\Qut(\Gr_i)$.
\end{thm}
\begin{proof}
Set $X_n =\{0,1, \ldots, n\}$, consider the classical graph $\Hr$ with $X_n$ as the set of vertices constructed in the proof of the previous proposition and denote by $B$ its adjacency matrix. Set $C(\X'):=C(\X)\otimes C(X_n)$ and
$$A':=\I_{\ell^2(\X)}\otimes B+\sum_{i=1}^n A_i\otimes P_i,$$
where $P_i$ is the projection on the $i$-th vertex in $\ell^2(X_n) \approx \C^{n+1}$. Note that for any operators $C_1, C_2$ on $\ell^2(\X)$ and $D_1, D_2$ on $\ell^2 (X_n) $ we have $(C_1 \otimes D_1) \bullet (C_2 \otimes D_2)
= (C_1 \bullet C_2) \otimes(D_1 \bullet D_2) $.

First, we should check that the formula above indeed defines the adjacency matrix of a quantum graph:
\begin{enumerate}
\item $A'=\bar A'$ -- follows from the fact that this holds for every summand;
\item $A'\bullet A'=A'$ -- follows from the fact that the summands are mutually orthogonal Schur-projections (as $P_i\bullet P_j=0$ for $i, j =1, \ldots, n$, $i \neq j$, and $\I_{\ell^2(\X)}\bullet A_i=0= A_i \bullet \I_{\ell^2(\X)}$), and the formula listed in the last paragraph;
\item $A'\bullet \I_{\ell^2(\X')}  =0$ -- this holds since $B\bullet \I_{\ell^2( X_n)}=0$ and $P_i\bullet \I_{\ell^2(X_n)}=0$.
\end{enumerate}

Now, let us study $\Qut(\X', A')$. Denote by $U$ its fundamental representation. We will continuously exploit Proposition \ref{prop:eigenvalues}.

Recall that in $\Hr$, each vertex $i\neq 0$ has degree $i$ and the vertex 0 has degree $\lceil n/2\rceil$, so $D_B=\lceil n/2\rceil P_0+\sum_{i=1}^n iP_i$. Now, we can compute
\begin{align*}
D_{A'}&=\I\otimes D_B+\sum_{i=1}^n d_i\,\I\otimes P_i=\I\otimes\left(D_B+\sum_{i=1}^n d_iP_i\right)\\
&=\I\otimes\left(\lceil n/2\rceil P_0+\sum_{i=1}^n(d_i+i)P_i\right),
\end{align*}
where $d_i$ is the degree of $\Gr_i=(\X,A_i)$, $i=1, \ldots,n$. Assume again that $d_1\le\cdots\le d_n$. All the coefficients are mutually distinct except that it could happen that $(d_{i_0}+i_0)=\lceil n/2\rceil$ for some $i_0<\lceil n/2\rceil$. Apart from this detail, the above formula forms the spectral decomposition of $D_{A'}$, so by Proposition~\ref{prop:eigenvalues}, we have that $P_i':=\I\otimes P_i$ is an intertwiner of $U$ for each $i\notin \{ 0,i_0\}$ and, consequently, $P_0'+P_{i_0}'$ is an intertwiner as well. We would like to show that $P_0'$ and $P_{i_0}'$ are separate intertwiners.

Notice that $\lceil n/2\rceil$ is connected to $0$, but not to $i_0$ in $\Hr$, so $P_{\lceil n/2\rceil}B(P_0+P_{i_0})=E_{\lceil n/2\rceil,0}$. Consequently, we have an intertwiner $P_{\lceil n/2\rceil}'A'(P_0'+P_{i_0}')=\I_{\ell^2(\X)}\otimes E_{\lceil n/2\rceil,0}$. Composing this intertwiner with its adjoint, we get the desired $P_0'$.

So, we achieved that $P_i'$ is an intertwiner for every $i=0,1,\dots,n$. Hence, ranges of these projections define invariant subspaces of $U$, so $U=\bigoplus_{i=0}^n U_{(i)}$.

Writing down the intertwiner relation for $A_i':=P_i'A'P_i'=\I_{\ell^2(\X)}\otimes A_i$, $i\neq 0$, we get $A_iU_{(i)}=U_{(i)}A_i$. So, each $U_{(i)}$ is a representation of $\Qut(\X, A_i)$.

Finally, take $i, j= 1, \ldots, n$, $i\neq j$ and study the intertwiner $P_i'A'P_j'=\I_{\ell^2(\X)}\otimes P_iBP_j$. Whenever there is an edge $j\to i$ in $\Hr$, this implies the relation $U_{(j)}=U_{(i)}$. Since the vertex $n$ is connected to everything in $\Hr$, this implies that all $U_{(i)}$ are equal, establishing $\Qut(\X', A') \subseteq \Qut(\X)\cap\Qut(\X, A_1)\cap\cdots\cap\Qut(\X,A_n)$

For the opposite inclusion, denote by $V$ the fundamental representation of $\bigcap_{i=1}^n\Qut(\X, A_i)$.  We need to show that $A'$ is an intertwiner of $U=V^{\oplus (n+1)}$. By definition, all $A_i$'s are intertwiners of $V$. Since the latter is a unitary representation, also the operators $A_i^{\ast}$ are intertwiners. Since $V^{\oplus (n+1)}$ is the tensor product of $V$ with the trivial representation on $\mathbb{C}^{n+1}$, any map on $\mathbb{C}^{n+1}$ is an intertwiner. It follows that $A'$ is an intertwiner of $V^{\oplus (n+1)}$. This implies that $\Qut(\X', A') =\Qut(\X)\cap\Qut(\X, A_1)\cap\cdots\cap\Qut(\X,A_n)$, which is what we wanted to prove.
\end{proof}

\subsection{Undirected version}

We will finally show that we can also obtain an undirected quantum graph with similar properties. Again, we briefly explain our construction classically. We use the same notation as in the previous subsection.

\begin{prop}
\label{P.classical22}
 Let $ N, n \in \N$ and let $\{\Gr_i\}_{i=1}^n$ be a set of directed regular graphs without loops on $N$ vertices . Then there exists an undirected graph $\Gr$ such that $\Aut \Gr=\bigcap_{i=1}^n \Aut \Gr_i$.
\end{prop}
\begin{proof}
In this case, we replace each point in $V$ by a graph $\tilde{\Hr}$ (an analogue of $\Hr$ of the previous subsection) with $2n+1$ vertices instead of just $n+1$. Thus $i,j>0$ are connected if and only if $i+j>2n$ and $0$ is connected to all $i>n$. The degree of $0$ is hence equal to $n$.

The graph $\Gr$ is constructed as in the proof of Proposition \ref{P.classical21}: we put an edge $(v,2i-1)-(w,2i)$ if there is an edge $v\to w$ in $\Gr_i$. The proof that $\Aut \Gr=\bigcap_{i=1}^n \Aut \Gr_i$ is then similar as in the directed version.
\end{proof}

Again, a quantum version follows a similar pattern.

\begin{thm}
\label{T.q22}
    Let $\X$ be a finite quantum set, let $n \in \N$ and let $\{\Gr_i\}_{i=1}^n$, $\Gr_i=(\X,A_i)$ be a set of regular quantum graphs without loops on $\X$. Then there exists an undirected  quantum graph $G=(\X',A')$ such that $\Qut (\Gr)=\bigcap_{i=1}^n\Qut(\Gr_i)$.
\end{thm}
\begin{proof}
We take the classical graph $\tilde{\Hr}$ as above, defined on the set $X_{2n} = \{0,1,\ldots, 2n\}$ and denote by $B$ the adjacency matrix of $\tilde{\Hr}$. Set $C(\X'):=C(\X)\otimes C(X_{2n})$ (of course, $C(X_{2n})\simeq\C^{2n+1}$), and
$$A':=\I_{\ell^2(\X)}\otimes B+\sum_{i=1}^n (A_i\otimes E_{2i-1,2i}+A_i^*\otimes E_{2i,2i-1}),$$
where this time  $E_{i,j} \in B(\ell^2(X_{2n}))$ denote the matrix units and $P_i = E_{i,i}$ is the projection on the $i$-th vertex in $\ell^2(X_{2n})$ ($i,j=0,1,\ldots, 2n$).

The verification that $A'$ defines a quantum graph is similar to the directed case in Theorem \ref{thm:directed-frucht}: 
\begin{itemize}
\item $A'=\bar{A'}$ -- follows from the fact that this holds for every summand;
\item $A'\bullet A'=A'$ -- follows from the fact that the summands are mutually orthogonal Schur-projections;
\item $A'\bullet \I_{\ell^2(\X')}=0$ -- this holds since $B\bullet \I_{\ell^2(X_{2n})}=0$ and $E_{ij}\bullet \I_{\ell^2(X_{2n})}=0$ if $i\neq j$.
\item $A'=A'^*$ -- since $B=B^*$ and $(A_i\otimes E_{2i-1,2i})^*=(A_i^*\otimes E_{2i,2i-1})$
\end{itemize}

Now, let us study $\Qut(\X', A')$. Denote by $U$ its fundamental representation.

Recall that in $\tilde{\Hr}$, each vertex $i\neq 0$ has degree $i$ and the vertex 0 has degree $n$, so $D_B=n P_0+\sum_{i=1}^n iP_i$. Now, we can compute
\begin{align*}
D_{A'}&=\I_{\ell^2(\X)}\otimes D_B+\sum_{i=1}^n d_i\,\I\otimes P_i=\I_{\ell^2(\X)}\otimes\left(D_B+\sum_{i=1}^{n} d_i(P_{2i-1}+P_{2i})\right)\\
&=\I_{\ell^2(\X)}\otimes\left(n P_0+\sum_{i=1}^n d_i((2i-1)P_{2i-1}+2iP_{2i})\right),
\end{align*}
where $d_i$ is the degree of $\Gr_i=(\X_i,A_i)$. All the coefficients are mutually distinct except that it could happen for some $k_0=2i_0$ or $k_0=2i_0-1$, $k_0<n$ that $(d_{i_0}+k_0)=n$. Nevertheless, we have that $P_k':=\I_{\ell^2(\X)}\otimes P_k$ is an intertwiner of $U$ for each $k\notin \{0,k_0\}$ and, consequently, $P_0'+P_{k_0}'$ is an intertwiner as well. We would like to show that $P_0'$ and $P_{k_0}'$ are separate intertwiners.

Notice that $n$ is connected to $0$, but not to $k_0$ in $\mathcal H$, so $P_nB(P_0+P_{k_0})=E_{n0}$. Consequently, we have an intertwiner $P_{n}'A'(P_0'+P_{k_0}')=\I_{\ell^2(\X)}\otimes E_{n0}$. Composing that with its adjoint, we get the desired $P_0'$.

So, we achieved that $P_k'$ is an intertwiner for every $k=0,1,\dots,2n$. Hence the ranges of these projections are invariant subspaces of $U$, so $U=\bigoplus_{i=0}^{2n} U_{(i)}$. Now, we would like to show that all $U_{(i)}$ coincide. Take any $0\le k\le 2n-2$ and $l=2n$ or $l=2n-1$. Then $P_k'A'P_l'=\I_{\ell^2(\X)}\otimes E_{kl}=E_{kl}'$ as $B_{kl}=1$ for all the choices considered. The intertwiner relation $E_{kl}'A'=A'E_{kl}'$ then exactly means that $U_{(l)}=U_{(k)}$, which is all we need: we proved that $U=V^{\oplus (2n+1)}$ for some unitary matrix $V$.

Finally, taking any $i\in\{1,\ldots,n\}$, we can compute $P_{2i-1}'A'P_{2i}=A_i\otimes E_{2i-1,2i}=:A_i'$. The intertwiner relation $UA_i'=A_i'U$ is then equivalent to the equality $VA_i=A_iV$, which is the final thing we needed to get $\Qut(\X', A')\subset\bigcap_{i=1}^n\Qut(\X, A_i)$.

For the opposite inclusion, denote by $V$ the fundamental representation of $\bigcap_{i=1}^n\Qut(\X, A_i)$. We need to show that $A'$ is an intertwiner of $U=V^{\oplus (2n+1)}$. By definition, all $A_i$'s are intertwiners of $V$. Since the latter is a unitary representation, also the operators $A_i^{\ast}$ are intertwiners. Since $V^{\oplus (2n+1)}$ is the tensor product of $V$ with the trivial representation on $\mathbb{C}^{2n+1}$, any map on $\mathbb{C}^{2n+1}$ is an intertwiner. It follows that $A'$ is an intertwiner of $V^{\oplus (2n+1)}$.  
\end{proof}

We now have all the tools in place to establish the quantum Frucht's theorem.  

\begin{thm}[Quantum Frucht's Theorem]\label{thm:qF}
Let $\G$ be a finite quantum group.  Then there is an undirected finite quantum graph $\Gr$ such that $\G \cong \Qut(\Gr)$.  That is, $\G$ acts faithfully on $\Gr$, and the canonical quotient map $C(\Qut(\Gr)) \to C(\G)$ is an isomorphism of Hopf *-algebras.
\end{thm}

\begin{proof} Combine Proposition \ref{P.quantum1} and 
Theorem \ref{T.q22}.
\end{proof}
\begin{remark}
We also reprove the weak quantum Frucht's theorem from \cite{DBRS24}, since if we start from a classical group $\Gamma$, then our construction yields a classical graph, whose quantum automorphism group is isomorphic to $\Gamma$.
\end{remark}
\begin{remark}
The proof of Theorem \ref{thm:qF} can be extended to general discrete quantum groups $\G$. Recall that quantum automorphism groups of connected, locally finite quantum graphs have been defined in \cite{Rol25}. We would have to alter our construction in the following way: we have an infinite family of quantum adjacency matrices, so we built an auxiliary infinite graph with all degrees of vertices different, so our algebra of functions would be $\ell^{\infty}(\G) \overline{\otimes} \ell^{\infty}(\mathbb{N})$. The formula for the quantum adjacency matrix would still be $A':=\I_{\ell^2(\X)}\otimes B+\sum_{i=1}^{\infty} (A_i\otimes E_{2i-1,2i}+A_i^*\otimes E_{2i,2i-1})$, which is well-defined on finitely supported elements and actually preserves this subspace, i.e. it is a locally finite quantum adjacency matrix (see \cite[Definition 3.24]{Was24}, \cite[Definition 4.1.1]{Rol25}). Since the action of the quantum automorphism group can be defined on the finitely supported elements (see \cite[Proposition 3.3.15]{Rol25}) and then extended to the multiplier algebra, and it commutes with the quantum adjacency matrix $A'$, it follows that the degree matrix $D_{A'}$ is an intertwiner, even though it can be unbounded. One can finish the proof exactly as in the finite case.

If $\G$ is finitely generated, i.e. $\widehat{\G}$ admits a finite generating set of representations, then we can keep the auxiliary classical graph finite, because it will be sufficient to choose the quantum adjacency matrices coming from the generating set. Then the quantum graph with quantum automorphism group $\G$ will be of bounded degree (see \cite[Definition 3.22]{Was24}).
\end{remark}

\section{Quantum Cayley graphs with few quantum automorphisms -- duals of classical groups} \label{sec:rigid-calyey}

In this section we will investigate whether it is possible to prove Frucht's theorem without the decolouring, namely by finding a quantum Cayley graph whose quantum automorphism group is the original quantum group. This problem has been studied by many authors and it has been conjectured by Watkins in \cite{Wat76} that it is possible for most countable groups, except for a large class of abelian groups, for generalized dicyclic groups and for finitely many other exceptions. This conjecture has been proved, first for finite groups, but also for finitely generated infinite groups fairly recently (see \cite{LdlS}). We will study this problem for finite quantum groups. We will show that for any non-abelian group $\Gamma$ one can find a quantum Cayley graph on its dual $\G:=\widehat{\Gamma}$ (so that $\cO(\G)$ is the group algebra of $\Gamma$) such that the quantum automorphism group is equal to $\G$. It will allow us to construct many examples of quantum graphs for which there is a huge difference between the classical automorphism group and the quantum automorphism group.

Let $\Gamma$ be a finite group, and let $\G=\widehat{\Gamma}$. A quantum Cayley graph on $\G$ is then determined by a projection $P$ in $C(\G) = \C[\Gamma]$. Consider the \emph{Fourier transform} of $P$, denoted $T:=\mathcal F(P)\in C(\Gamma)$ (note that in this case the Fourier transform takes a very simple form: if $P=\sum_{g\in \Gamma} c_g\lambda_g$ for some scalars $c_g\in \C$, then $T(g) = c_g$ for each $g \in \Gamma$). The associated quantum adjacency matrix $A_P$ is then given by the Herz-Schur multiplier
\[ m_T(\lambda_g) = T(g) \lambda_g, \;\;\; g \in \Gamma.\]
Thus computing the relevant quantum symmetry group amounts (see also Proposition \ref{prop:genformGamma} below) to determining the conditions on the operator-valued unitary matrix $(a_{g,h})_{g, h \in \Gamma} \in M_{|\Gamma|}(C)$ (over some C*-algebra $C$) which would assure that the formula
\[ \alpha(\lambda_h):= \sum_{g \in \Gamma} \lambda_g \otimes a_{g,h}, \;\;\; h\in \Gamma, \]
defines a trace preserving morphism $\alpha:\C[\Gamma] \to \C[\Gamma] \otimes C$, commuting with the multiplier $m_T$ in following sense: $\alpha \circ m_{T} = (m_{T}\otimes \id)\circ \alpha$. The latter condition translates into the requirement that
\[ a_{g,h} = 0 \;\;\textup{ whenever } \;\;g, h \in \Gamma, \; T(g) \neq T(h),\]
so it depends only on the level sets of $T$ (and not on its actual values). 
This means that the computation of the relevant quantum symmetry group follows the same general scheme as the computations of quantum isometry groups of group duals, as initiated in \cite{BhS10} (where $T$ is generally the length function for a given generating set).

In particular, if $T$ is one-to-one, then the quantum automorphism group is exactly $\G$. Moreover, the Fourier transform of a rank one projection $|\xi\rangle\langle \xi|$ inside a block of $\mathbb{C}[\Gamma]$ corresponding to an irreducible representation $\pi$ of $\Gamma$ is (up to normalization) the matrix coefficient $g \mapsto \langle \xi, \pi(g) \xi\rangle$. 

\begin{prop}\label{prop:faithfulirrep}
Let $\Gamma$ be a finite group and suppose that $\pi$ is a faithful, irreducible representation of $\Gamma$ on $\HH$. Then there exists $\xi \in \HH$  such that the matrix coefficient $g\mapsto \langle \xi, \pi(g) \xi\rangle$ separates the points of $\Gamma$; note that  the rank one projection $|\xi\rangle\langle \xi|$ belongs to $\C[\Gamma]$.
\end{prop}
\begin{proof}
Fix a pair $g,h \in \Gamma$ such that $g\neq h$. The set $\HH_{g,h}:=\{\xi \in \HH: \langle \xi, (\pi(g) - \pi(h)) \xi\rangle =0\}$ is Zariski closed (as a zero set of a quadratic function), so its complement is open and dense as soon as it is non-trivial. As $\pi$ is faithful, $\pi(g) - \pi(h)\neq 0$, so for some $\xi \in \HH$ we have $\langle \xi, (\pi(g) - \pi(h))\xi\rangle \neq 0$ by the polarisation identity. It follows that the intersection $\K:=\bigcap_{g\neq h} \HH \setminus \HH_{g,h}$ is an open, dense subset of $\HH$ and for every $\xi \in \K$ the matrix coefficient function $ g \mapsto \langle \xi, \pi(g) \xi\rangle$ separates the points of $\Gamma$.
\end{proof}
This proposition together with Proposition \ref{prop:fullalgebra} can already be applied to some classes of groups, such as the permutation groups, dihedral groups, and all finite simple groups (as all non-trivial representations of simple groups are faithful).  This result should be contrasted with the work of Banica--McCarthy \cite{BM22}, which shows for example that $\widehat S_3$ cannot be realized as $\Qut(\Gr)$ for any {\it classical} graph $\Gr$.   

We will now prove the main result for general non-abelian groups $\Gamma$. In general, we will not obtain multipliers that separate all points of $\Gamma$, but they will separate sufficiently many points to yield the desired conclusion. The result will utilize the fact that the elements $a_{g,h}$ are not independent, because we want $\alpha: \C[\Gamma] \to \C[\Gamma] \otimes C$ to be a $\ast$-homomorphism. On the one hand, we have
\[
\alpha(\lambda_{h_1} \lambda_{h_2}) = \sum_{g_1, g_2} \lambda_{g_1 g_2} \otimes a_{g_1, h_1} a_{g_2, h_2}.
\]
On the other hand, this expression is equal to 
\[
\alpha(\lambda_{h_1 h_2}) = \sum_{g} \lambda_g \otimes a_{g, h_1 h_2}.
\]
If we rearrange the first expression and recall that the $\lambda_g$'s are linearly independent, it follows that
\[
a_{g, h_1 h_2} = \sum_{g_1} a_{g_1, h_1} a_{g_1^{-1}g, h_2}.
\]

Let us now consider the case of a general non-abelian group $\Gamma$. We will take a look at the representation $\pi:=\bigoplus_{k=1}^l \alpha_k$, defined as the direct sum of all irreducible representations that are not one-dimensional (i.e.\ not characters); their carrier Hilbert spaces are $\mathbb{C}^{n_{k}}$, with unit spheres $S^{n_{k}-1}$. In each of the Hilbert spaces we will pick a unit vector $\xi_k$, giving us a rank one projection. The Fourier transform of the sum of the relevant projections will be equal to $T_{\bm{\xi}}(g) := \sum_{k} \langle \xi_k, \alpha_k(g) \xi_k\rangle$. For each pair $g,h \in \Gamma$, say $g \neq h$, we will consider the algebraic set $V_{g,h}^{\bm{\xi}}:= \{\bm{\xi} \in \prod_{k=1}^l S^{n_{k}-1}: T_{\bm{\xi}}(g) = T_{\bm{\xi}}(h)\}$. It is an algebraic subset of an irreducible variety (a product of spheres, we are working over $\mathbb{R}$, not over $\mathbb{C}$), so it is either nowhere dense, or equal to the whole space. Suppose that it is equal to the whole space. It means that for all unit vectors $\xi_k \in \mathbb{C}^{n_k}$ as above we get
\[
\sum_{k=1}^l \langle \xi_k, (\alpha_k(g)-\alpha_k(h))\xi_k\rangle = 0.
\]
Since it is a sum of functions depending on different variables, all summands must be constant, i.e.\ there exists $c_k \in \mathbb{C}$ such that $\langle \xi_k, (\alpha_k(g) - \alpha_k(h)) \xi_k\rangle = c_k$ for all $\xi_k \in S^{n_k -1}$. It follows by rescaling that $\langle \xi_k, (\alpha_k(g) - \alpha_k(h) - c_k \mathds{1}) \xi_k\rangle = 0$ for any $\xi_k \in \mathbb{C}^{n_k}$, therefore
\[
\alpha_k(g) - \alpha_k(h) = c_k \mathds{1}
\]
for every $k =1, \ldots, l$.

Let now $\rho$ be the direct sum of all irreducible representations of $\Gamma$. We have $C^{\ast}_{\rho}(\Gamma)\simeq \mathbb{C}[\Gamma]$, because $\rho$ is the left regular representation with multiplicities ignored (or $\pi$ with all the one-dimensional irreducible representations added). The condition $\alpha_k(g) - \alpha_k(h) = c_{k}\mathds{1}$ for all $k =1,\ldots, l$ translates to the statement that $\rho(g) - \rho(h)$ belongs to the center of $C^{\ast}_{\rho}(\Gamma)$. Forgetting about $C^{\ast}_{\rho}(\Gamma)$ and working directly with the group algebra $\mathbb{C}[\Gamma]$, we obtain for each $t\in G$ that $tgt^{-1} - tht^{-1} + h - g=0$. Since group elements are linearly independent in $\mathbb{C}[\Gamma]$ (as they are orthogonal in $\ell^{2}(\Gamma)$), and we were only interested in pairs $(g,h)$ with $g\neq h$, it follows that $tgt^{-1} = g$ and $tht^{-1}=h$ for $t\in \Gamma$, i.e. $g$ and $h$ both belong to the center of $\Gamma$. It follows that for almost all choices of $\bm{\xi}$ we have $a_{g,h}=0$ as soon as $g\neq h$ and one of the elements $g$ or $h$ is not central. We therefore proved the following.
\begin{prop}
If we have two group elements $g\neq h\in \Gamma$ such that at least one of them is non-central then the set $V^{\bm{\xi}}_{g,h}$ is nowhere dense, that is for almost all choices of $\bm{\xi}$ we have $T_{\bm{\xi}}(g) \neq T_{\bm{\xi}}(h)$.
\end{prop}
We are now ready to state and prove the main result of this section.

\begin{thm}\label{Thm:cayleyaut}
Let $\Gamma$ be a finite non-abelian group and set $\G=\widehat{\Gamma}$. Then there exists a projection $P$ in $C(\G)$ such that the quantum automorphism group of the resulting quantum Cayley graph $\Qut(\G, A_P)$ is equal to $\G$.
\end{thm}
\begin{proof}
We already know from the previous proposition that for a generic choice of $\bm{\xi}$ we have $a_{g,h}=0$ whenever either $g$ or $h$ is non-central.   

Suppose now that both $g$ and $h$ belong to the center of $\Gamma$. Since we assume that $\Gamma$ is non-abelian, there exists an element $t$ that does not belong to the center. In the formula $a_{g, h_1 h_2} = \sum_{g_1} a_{g_1, h_1} a_{g_1^{-1}g, h_2}$ we choose $h_1:=t$ and $h_2 := t^{-1}h$. Then for each $g_1\in \Gamma\setminus\{t\}$ we have $a_{g_1,h_1} = a_{g_1,t}=0$ because $t$ is not from the center and $g_1\neq t$. If $g_1=t$ then the elements $g_{1}^{-1}g = t^{-1}g$ and $h_2 = t^{-1}h$ are non-central and distinct, so $a_{t^{-1}g, t^{-1}h}=0$, hence $a_{g,h}=0$. If follows that $a_{g,h}=0$ for $g\neq h$, hence the quantum automorphism group is precisely $\G$.
\end{proof}

The quantum Cayley graphs constructed in the proof above are highly non-central. Note that in general we cannot expect a version of the above result for central quantum Cayley graphs, as noted in Proposition \ref{prop:noncentral}.

Finally note that using the above construction we can produce examples of quantum graphs with a big gap between the classical and quantum automorphism groups.

\begin{thm}\label{thm:trivquat}
Suppose that $\Gamma$ is a finite perfect group (i.e.\ it has no non-trivial abelian quotients), $\G= \widehat{\Gamma}$. Then there exists a quantum graph structure on $C(\G)$ such that the classical automorphism group is trivial, but the quantum automorphism group acts transitively.
\end{thm}
\begin{proof}
It follows from the assumption that the only character of $\Gamma$ is the trivial representation. It follows that $\G$ has no non-trivial classical subgroups, as these would correspond to Hopf $^*$-algebra quotients of $\C[\Gamma]$ and the only option there is $\C$. By Theorem \ref{Thm:cayleyaut} we can find a projection $P \in C(\G)$ such that the quantum automorphism group of the resulting quantum Cayley graph is equal to $\G$. Since $\G$ has no non-trivial classical subgroups, it follows that the classical automorphism group is trivial. Since the action of $\G$ on itself is given by the comultiplication, it is transitive (i.e.\ ergodic).   
\end{proof}

\subsection{Examples of quantum Cayley graphs and their quantum symmetry groups on duals of symmetric groups}

In this section we will describe some examples related to duals of finite permutation groups, essentially following the connection to quantum isometry groups mentioned above. We begin however with some simple remarks concerning the general context.

Let $\Gamma$ be a finite group. The discussion above shows that to compute the quantum symmetry groups of the quantum Cayley graph associated to a projection $P$ in $\C[\Gamma]$ we need to first understand the level sets of $T=\mathcal F(P)$, and then to compute the relevant quantum symmetry group `preserving the level sets'. 

To address the first point, we may apply the following easy lemma, which is the direct consequence of the Peter-Weyl theory (for finite groups).

\begin{lem} \label{PeterWeyl}
Consider the  isomorphism $\rho:\C[\Gamma]\to \bigoplus_{\pi \in \textup{Irr}(\Gamma)} M_{n_{\pi}}$, given by the (inverse Fourier) identification $\lambda_g \mapsto \bigoplus_{\pi \in \textup{Irr}(\Gamma)} \pi(g)$. Fix $\pi \in \textup{Irr}(\Gamma)$ and vectors $\xi,\eta \in H_\pi$, and set $x=|\xi\rangle \langle \eta | \in M_{n_{\pi}}$. Then 
\[ \rho^{-1} (x) = \frac{ n_\pi}{|\Gamma|}\sum_{g\in \Gamma}
\langle \pi(g) \eta,  \xi \rangle\,\lambda_g.\]
Further if $P=\id_{H_\pi}$ then we have
\[ \rho^{-1} (P) = \frac{ n_\pi}{|\Gamma|}\sum_{g\in \Gamma} \Tr(\pi(g)^*) \, \lambda_g.\] 
\end{lem}
\begin{proof}
Set $n=n_\pi$ and fix an orthonormal basis $(e_1, \ldots, e_n)$ of $H_\pi$.
By linearity it suffices to prove the result for $\xi=e_i, \eta = e_j$ for some $i,j=1,\ldots,n$. Set $y=\frac{ n_\pi}{|\Gamma|}\sum_{g\in \Gamma}
\langle \pi(g) e_j,  e_i \rangle\,\lambda_g$. We need to check that $\rho(y) =x$, so that
\[ \rho(y)_\pi = |e_i\rangle \langle e_j|,\]
\[ \rho(y)_{\pi'} =0, \;\;\; \textup{for} \;\pi' \in \textup{Irr}(\Gamma), \pi'\neq \pi.\]
But the two formulas displayed above are equivalent to the following:
\[ \frac{ n_\pi}{|\Gamma|}\sum_{g \in \Gamma}
\overline{ \langle e_i,  \pi(g) e_j \rangle} \langle e_k,  \pi(g) e_l \rangle = \langle e_k, e_i \rangle \langle e_j, e_l \rangle,    \]
\[ \frac{ n_\pi}{|\Gamma|}\sum_{g\in \Gamma}
\overline{ \langle e_i,  \pi(g) e_j \rangle} \langle \zeta,  \pi'(g) \kappa \rangle = 0  \]
for all $k,l=1,\ldots,n_\pi$ and all vectors $\zeta, \kappa \in H_{\pi'}$, where $\pi' \in \textup{Irr}(\Gamma), \pi'\neq \pi$. These are however precisely the Peter-Weyl orthogonality relations.

The second formula follows from the first one by linearity.

\end{proof}

We will now make a few comments on  `quantum actions on $\C[\Gamma]$' preserving the trace and a certain partition of $\Gamma$.

\begin{defn}
Let $\Gamma$ be a finite group, $F \subset \Gamma$, let $A$ be a unital C*-algebra, and let $\alpha:\C[\Gamma] \to \C[\Gamma] \otimes A$ be a unital *-homomorphism. We say that $\alpha$ preserves $F$ if for any $g\in F$ we have
\[ \alpha (\lambda_g) = \sum_{g' \in F} \lambda_{g'} \otimes A_{g,g'} \]
for some $A_{g,g'} \in A$.
\end{defn}

Note that in general $\alpha$ as above can be always described via a matrix $[A_{g, g'}]_{g,g' \in\Gamma}\in M_{|\Gamma|}(A)$; moreover every unital homomorphism preserves $\{e\}$.

\begin{prop}\label{prop:genformGamma}
    Suppose that $\Gamma$ is a finite group,  $A$ a unital $\Cst$-algebra, and let $\alpha:\C[\Gamma] \to \C[\Gamma] \otimes A$ be a unital *-homomorphism which is trace preserving, i.e. $(\tau \otimes id)\circ \alpha = \tau(\cdot)1_A$. Then 
    \begin{itemize}
        \item[(i)] the matrix $U:=[A_{g,g'}]_{g,g' \in\Gamma}\in M_{|\Gamma|}(A)$ describing $\alpha$ is unitary;
         \item[(ii)] for any $F \subset \Gamma$ if $\alpha$ preserves $F$ and $\Gamma \setminus F$, then the submatrix $U_F:=[A_{g,g'}]_{g,g' \in F}\in M_{|F|}(A)$ is also unitary.
    \end{itemize}
\end{prop}
\begin{proof}
    The first part is standard, see for example \cite[Proposition 3.3.13]{ASS17}. 
    
    The second is an easy consequence of the fact that if an operator-valued unitary matrix is block-diagonal, the blocks themselves must be unitary. Note that using \cite[Theorem 3.3.8 (i)]{ASS17} one can show that it would be in fact sufficient to assume that $\alpha$ preserves $F$; we will however not need it here.  
\end{proof}

Let us now apply these arguments to finite symmetric groups. 
Consider the defining representation $\pi$ of $S_n$ on $\C^n$ (it is not irreducible, but this will not play a role in the arguments below). The character of the defining representation counts (up to rescaling) the number of fixed points of a given permutation: given $\sigma \in S_n$
we have 
\[\Tr(\pi(\sigma)) = \sum_{i=1}^n \langle e_i, e_{\sigma(i)} \rangle = |\textup{Fix}(\sigma)|.\]

This already allows us to formulate the first result. In the formula below $D[ \widehat{S_n}]$ denotes the quantum group associated to the `doubling' of the Hopf algebra $\C[S_n]$ (see for example \cite{LDS12} or \cite{SkS14}). For a permutation $\sigma \in S_n$ we will also write simply $\sigma \in \mathbb{C}[S_n]$ (and not $\lambda_\sigma$).

\begin{prop} \label{prop:Sndifferent}
Consider $\Gamma=S_n$, $n\geq 3$. Let $P$ be the (central) projection in $\C[S_n]$ corresponding to the character of the (irreducible part of the) defining representation -- which is generating -- and denote by $A$ the respective quantum adjacency matrix. 
Then 
\[ \widehat{S_n}  \subsetneq D[ \widehat{S_n}] \subsetneq \Qut(\widehat{S_n}, A).  \]
\end{prop}

\begin{proof}
In \cite{LDS12} the authors -- following the earlier work in \cite{BhS10} for $n=3$ -- compute the quantum isometry group of $\widehat{S_n}$ for the generating set $T$ consisting of the first $n-1$ `consecutive' transpositions (i.e.\ $T=\{(12), (23),\ldots, (n-1, n)\}$) and show that this coincides with the `doubling' $D(\widehat{S_n})$. 

In particular  \cite[Formula (2.33)]{LDS12}  shows that the action   preserving the set $T$ must necessarily be of the form 
\begin{equation} \label{Doublingaction} \alpha(s_i) = s_i \otimes \sigma_i   + s_{n-i}  \otimes \tau_i , \; i=1,\ldots, n-1,\end{equation}
where $s_i=(i,i+1)$ and elements $\sigma_i, \tau_j$ denote the corresponding elements in the (mutually orthogonal!) copies of $\C[S_n]$ inside the doubling Hopf algebra. Thus 
\begin{align*}&\alpha((n1)) = \alpha(s_1) \cdots \alpha(s_{n-1}) \alpha(s_{n-2}) \cdots \alpha (s_1) \\&= 
s_1 \cdots s_{n-1} s_{n-2} \cdots s_1 \otimes  \sigma_1 \cdots \sigma_{n-1} \sigma_{n-2} \cdots \sigma_1 + 
s_{n-1} \cdots s_{1} s_{2} \cdots s_{n-1}  \otimes 
\tau_1 \cdots \tau_{n-1} \tau_{n-2} \cdots \tau_1,
\end{align*}
so in fact
\[\alpha((n1)) = (n1)  \otimes X.\]
Hence the action preserves also the singleton set $(n1)$, and further the set of all transpositions. This already shows that 
$D[ \widehat{S_n}] \subset \Qut(\widehat{S_n}, A)$. Consider now the universal action preserving the set $T'=\{s_2,\ldots, s_{n-1}, (n1)\}$. By repeating the same arguments as above (essentially reordering the indices) we can see that this would also preserve the set of all transpositions, but not the set $\{s_1, \ldots, s_{n-1}\}$ (as the `orbit' of say $s_2$ would include also $s_n$). Thus the inclusion $D[ \widehat{S_n}] \subset \Qut(\widehat{S_n}, A)$ must be strict.
\end{proof}

We will now specify the situation to $n=3$, again consider the defining representation, but allow now an arbitrary rank one-projection, determined by a norm $1$ vector $(\alpha, \beta, \delta) \in \C^3$, which is orthogonal to the `symmetric subspace', i.e. $\alpha +\beta +\delta =0$). As we only care about the level sets of $T$, Lemma \ref{PeterWeyl} implies that in fact we can ignore the normalization. Set then $\xi = (1, \alpha, -1-\alpha)$, $x= |\xi \rangle \langle \xi |$, and $T=\rho^{-1}(x)$, using the notation of Lemma \ref{PeterWeyl}.
We then have, by the same lemma,
\[T_{\mathrm{id}} = 1+|\alpha|^2 +|1+\alpha|^2,\]
\[ T_{(12)} = |1+\alpha|^2 +\alpha + \bar{\alpha} , 
T_{(13)} = |\alpha|^2 - 2 -(\alpha + \bar{\alpha}),
T_{(23)} = 1- (\alpha + \bar{\alpha}) - 2 |\alpha|^2  ,  \]
\[ T_{(123)} = - 1 - |\alpha|^2  - \bar{\alpha} - 2 \alpha, \; T_{(132)} = - 1 - |\alpha|^2  - 2\bar{\alpha} -  \alpha.\]
It is now easy to see the following facts:
\begin{itemize}
\item[(1)] if $\alpha=0$,  the level sets of $T$ are singletons apart from $\{(123), (132)\}$ and $\{(12),(23)\}$;
\item[(2)] if $\alpha$ is non-zero, but `sufficiently' close to 0 and real, the level sets of $T$ are singletons, apart from $\{(123), (132)\}$;
\item[(3)] if $\alpha$ is non-zero, but `sufficiently' close to 0 and purely imaginary, the level sets of $T$ are singletons.
\end{itemize}

Thus we obtain the following proposition.

\begin{prop}\label{prop:S3}
Consider $\Gamma=S_3$. Let $P_0$ be the (central) projection in $\C[S_3]$ corresponding to the character of the (irreducible part of the) defining representation,  let $P_i$ denote the rank one projections in $\C[S_3]$ corresponding to the cases listed before the proposition ($i=1,2,3$). Denote by $A_i$ the respective quantum adjacency matrices  ($i=0,1,2,3$).
Then 
\[ \widehat{S_3} = \Qut(\widehat{S_3}, A_3) = \Qut(\widehat{S_3}, A_2)  \subsetneq \Qut(\widehat{S_3}, A_1) = D(\widehat{S_3}) \subsetneq \Qut(\widehat{S_3}, A_0).  \]
\end{prop}
\begin{proof}
The first equality is a consequence of the discussion before Proposition \ref{prop:faithfulirrep}. The third equality ($\Qut(\widehat{S_3}, A_1) = D(\widehat{S_3})$) follows from the fact that the level sets for the relevant multiplier correspond precisely to these arising from the length function considered in \cite[Theorem 4.2]{BhS10} and the statement of that theorem.
The behaviour of multiplier sets described before the formulation of the statement we are proving implies that $\Qut(\widehat{S_3}, A_2)$ must be a quantum subgroup of $\Qut(\widehat{S_3}, A_1)$. But inspecting the action formula for $D(\widehat{S_3})$ given in \eqref{Doublingaction} for $i=1$ shows that the fact that the action preserves the singleton level set $\{s_1\}$ forces the second factor in the quotient of the doubling to disappear. Thus we end up with $\widehat{S_3} = \Qut(\widehat{S_3}, A_2)$.

Finally the fact that the inclusions are strict follows from a direct inspection in the first case and from Proposition \ref{prop:Sndifferent} in the second case.
\end{proof}

The above proposition allows us to deduce that central quantum Cayley graphs in general cannot suffice for a quantum version of Watkins Theorem. 

\begin{prop} \label{prop:noncentral}
The quantum group $\G=\widehat{S_3}$ does not admit any \emph{central} quantum Cayley graph $A$ such that $\G= \Qut(C(\G), A)$. 
\end{prop}
\begin{proof}
Let us write explicitly $\mathbb{C}[S_3] \approx M_2 \oplus \C \oplus \C$, with the first factor given by the canonical irreducible representation, the second corresponding to the sign representation, and the third to the trivial one. Denote the respective non-zero minimal central projections in $\mathbb{C}[S_3]$ by $P$, $Q$ and $R$.

Proposition \ref{prop:S3} shows that if we work with $P$, we have   $\G \subsetneq\Qut(\widehat{S_3}, A_0)$ (using the notation above, where $P=A_0$). 

On the other hand the multiplier functions corresponding to $R$ (i.e.\ $\delta_e$) and to $P+Q+R$ (i.e.\ constant function equal $1$) do not give any  extra restrictions regarding the quantum group action.
Thus if we consider the last projection we see that $\Qut(\widehat{S_3}, R) \approx \Qut(M_2 \oplus \C \oplus \C)$, which is even larger then $\Qut(\widehat{S_3}, A_0)$. The same happens when we replace $R$ by $P+Q+R=I$ or by $P+Q=I-R$. 

Finally by the same argument $\Qut(\widehat{S_3}, Q) = \Qut(\widehat{S_3}, Q+R)  = \Qut(\widehat{S_3}, P+R) = \Qut(\widehat{S_3}, A_0)$ (as $Q+R= I-P$).
\end{proof}

\smallskip
\noindent {\bf Acknowledgments. } M.B.\ was partially supported by an NSERC Discovery Grant.  A.S.\ was partially supported by the National Science Center (NCN) grant no. 2020/39/I/ST1/01566.
J.M.\ was supported by JSPS KAKENHI Grant Number JP23KJ1270. M.W\ was partially supported by the National Science Center, Poland (NCN) grant no. 2021/43/D/ST1/01446. 
The project is co-financed by the Polish National Agency for Academic Exchange within the Polish Returns Programme. 

We acknowledge fruitful discussions with Roland Vergnioux and Moritz Weber at various stages of the project and thank the anonymous referee for careful reading of our article and useful comments improving the presentation. 

\vspace{5 pt}
\includegraphics[scale=0.5]{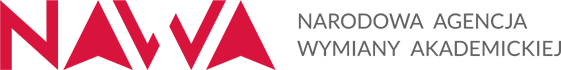}

\printbibliography

\end{document}